\documentclass[11pt]{article}

  \usepackage{a4wide}
  \usepackage{amsmath}
  \usepackage{amssymb}
  \usepackage{latexsym}
  \usepackage[pdftex]{graphicx}
  \usepackage{subcaption}
  \usepackage{color}
  \usepackage[mathscr]{euscript}
  \usepackage{url}

  \newenvironment{proof}{\vspace{1ex}\noindent{\bf Proof.}}{\hspace*{\fill}$\blacksquare$\vspace{1ex}}
  \newenvironment{proofof}[1]{\vspace{1ex}\noindent{\bf Proof of #1.}}{\hspace*{\fill}$\blacksquare$\vspace{1ex}}

  \newtheorem{theorem}{Theorem} 
  \newtheorem{lemma} [theorem] {Lemma}
  \newtheorem{corollary} [theorem] {Corollary}
  \newtheorem{proposition} [theorem] {Proposition}
  
\newcommand{\Ncal}[0]{\ensuremath{{\mathcal N}}}

\newcommand{\Vcal}[0]{\ensuremath{{\mathcal V}}}

\newcommand{\Xcal}[0]{\ensuremath{{\mathcal X}}}
\newcommand{\Ycal}[0]{\ensuremath{{\mathcal Y}}}
\newcommand{\Zcal}[0]{\ensuremath{{\mathcal Z}}}
\newcommand{\eR}[0]{\ensuremath{ \mathbb R}}

\newcommand{\Qu}[0]{\ensuremath{ \mathbb Q}}
\newcommand{\eN}[0]{\ensuremath{ \mathbb N}}

\newcommand{\norm}[1]{\ensuremath{\|#1\|}}

\newcommand{\Bee}[0]{\ensuremath{{\mathbb B}}}
\newcommand{\Pee}[0]{\ensuremath{{\mathbb P}}}
\newcommand{\Haa}[0]{\ensuremath{{\mathbb H}}}
\newcommand{\Ee}[0]{\ensuremath{{\mathbb E}}}

\newcommand{\isd}[0]{\hspace{.2ex} \raisebox{-.1ex}{$=$} \hspace{-1.5ex} 
\raisebox{1ex}{{$\scriptstyle d$}} \hspace{.8ex} }

 \newcommand{\eps}{\varepsilon}

\newcommand{\orig}{\underline{0}}
\newcommand{\BGab}[0]{\ensuremath{B_{\text{Gab}}}}
\newcommand{\pijl}[0]{\ensuremath{\leftrightsquigarrow}}
\usepackage[cmtip,all]{xy}
\newcommand{\langepijl}{\xymatrix{{}\ar@{<~>}[r]&{}}}

\DeclareMathOperator{\dist}{dist}

\DeclareMathOperator{\Po}{Po}
\DeclareMathOperator{\Bi}{Bi}

\DeclareMathOperator{\dd}{d}

\DeclareMathOperator{\arcsinh}{arcsinh}
\newcommand{\tdist}[0]{\ensuremath{\widetilde{\dist}}}

\DeclareMathOperator{\kone}{cone}

\definecolor{orange}{RGB}{255,127,0}
\definecolor{pink}{RGB}{255,150,150}

\definecolor{darkgreen}{rgb}{0,0.4,0}

\begin{document}

\title{Non-vanishing uniqueness threshold for\\hyperbolic Poisson-Voronoi percolation\\in dimension at least three}

\author{%
Matthias Irlbeck\thanks{Bernoulli Institute, University of Groningen, The Netherlands. E-mail: {\tt m.irlbeck@rug.nl}.}
\and
Tobias M\"uller\thanks{Bernoulli Institute, University of Groningen, The Netherlands. E-mail: {\tt tobias.muller@rug.nl}.} 
}

\date{\today}

\maketitle

\begin{abstract} 
We study the threshold for the existence of {\em exactly one} unbounded cluster for 
Poisson-Voronoi percolation on the $d$-dimensional hyperbolic space $\Haa^d$ for $d\geq 3$.
By recent results of Greb\'ik and Recke~\cite{GrebikRecke} and d'Achille et al.~\cite{dAchilleVanishing}, 
this ``uniqueness threshold'' $p_u(\lambda)$ tends to zero as the intensity $\lambda$ of the underlying Poisson point 
process tends to zero, 
for Poisson-Voronoi percolation defined on an ambient space
from a family of geometric spaces that includes Cartesian products $\Haa^{d_1}\times\dots\times\Haa^{d_k}$ with 
$k, d_1,\dots,d_k \geq 2$.
In contrast, for Poisson-Voronoi
percolation on the hyperbolic plane $\Haa^2$ Benjamini and Schramm~\cite{BS2001} have shown that $p_u(\lambda)$ tends to one as 
$\lambda$ tends to zero, and $p_u(\lambda)>1/2$ for all $\lambda>0$. 
An unpublished argument of D'Achille and Curien shows that 
for Poisson-Voronoi percolation on $\Haa^d$ with $d\geq 3$, the uniqueness threshold 
satisfies $p_u(\lambda) \leq 1/2$ for all $\lambda>0$. 

Here we will show that $\inf_{\lambda>0} p_u(\lambda) > 0$ for Poisson-Voronoi percolation on $\Haa^d$ with $d\geq 3$.
This answers a question of Greb\'ik and Recke~\cite{GrebikRecke}.

The {\em ideal Poisson-Voronoi tessellation} is a random tessellation of $\Haa^d$ introduced recently by D'Achille et al.~\cite{dAchilleIdeal}. In a sense that can be made precise, it corresponds to the $\lambda \searrow 0$ limit of the Poisson-Voronoi tessellation. Via a minor adaptation of our proof, we are also able to show that the uniqueness threshold of the ideal Poisson-Voronoi tessellation is strictly positive.
\end{abstract}

\section{Introduction and statement of results.}

We consider Poisson-Voronoi percolation on the $d$-dimensional hyperbolic space $\Haa^d$.
That is, with each point $z$ of a Poisson point process $\Zcal$ of constant intensity $\lambda$ on $\Haa^d$ we associate its
{\em Voronoi cell}:

$$ C(z) := \left\{ x \in \Haa^d : \dist(x,z) \leq \dist(x,w) \text{ for all $w \in \Zcal$}\right\}, $$

\noindent
where $\dist(.,.)$ denotes the hyperbolic metric. (See Sections~\ref{sec:hyperbolicgeom},~\ref{sec:hyperPPP} for 
the precise definitions of hyperbolic Poisson point processes, hyperbolic distance, etc.)
We colour these Voronoi cells either black or white according to independent $p$-biased coin flips (black has probability $p$, 
white has probability $1-p$).
Figure~\ref{fig:sim} renders a computer simulation of this process on $\Haa^3$, depicted in the Poincar\'e halfspace 
model of $\Haa^3$.%
\begin{figure}[h!]
\begin{center}
\includegraphics[width=.7\textwidth]{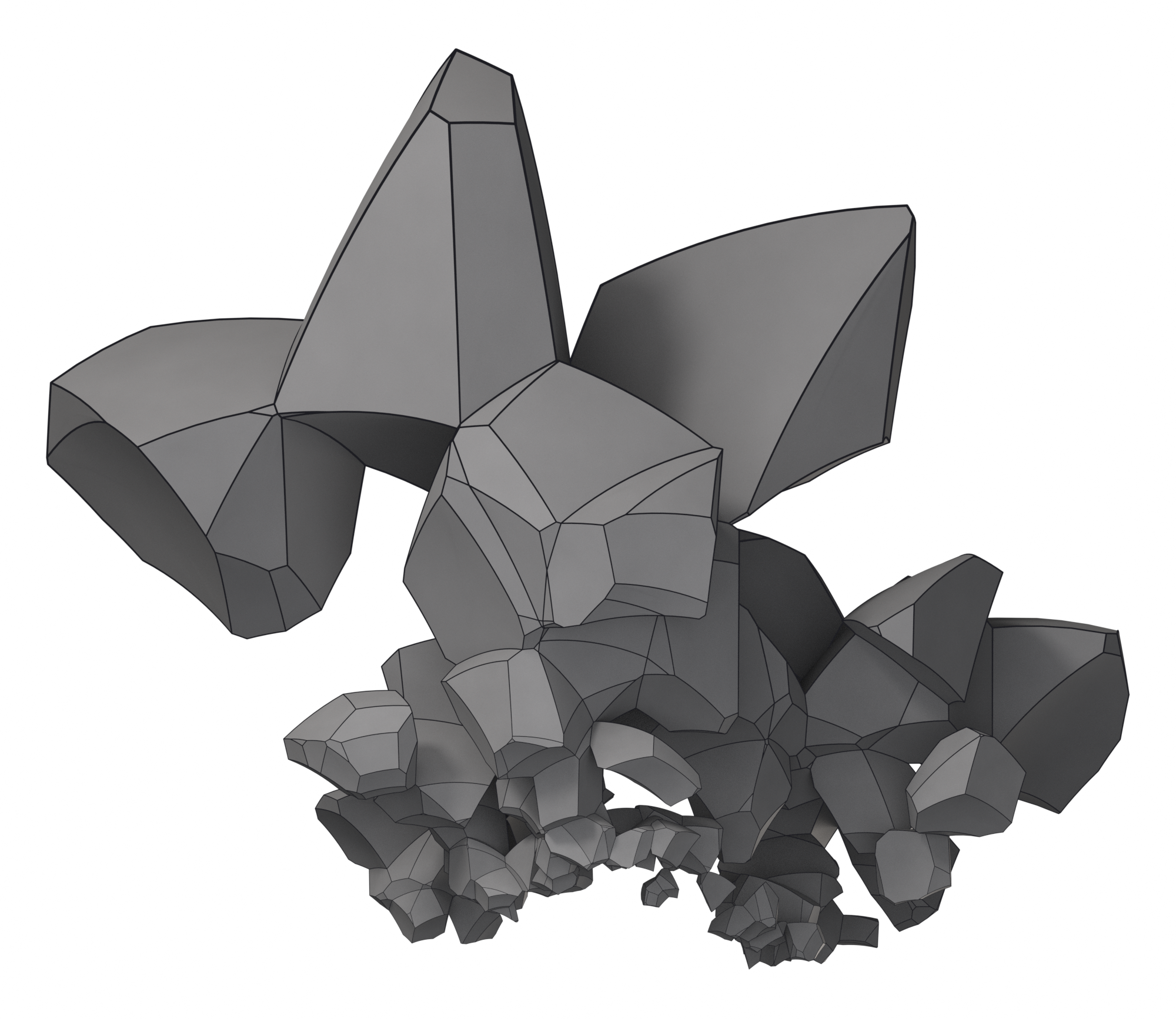}
 \caption{Computer simulation of Poisson-Voronoi percolation on $\Haa^3$.
 A single cluster is shown, depicted on the Poincar\'e halfspace model.
 \label{fig:sim}}
 \end{center}
\end{figure}
We say that two points $x,y \in \Haa^d$ (not necessarily belonging to $\Zcal$) 
belong to the same {\em (black) cluster}, denoted
$x \pijl y$, if there is a continuous curve connecting $x$ and $y$
that stays entirely in the union of the black Voronoi cells.
(Formally speaking, a black cluster is an inclusion maximal subset $A$ of the union of the black cells
with the property that $x\pijl y$ for all $x,y \in A$.)
We say that {\em percolation} occurs if there is an unbounded (wrt.~$\dist$) black cluster.
A pertinent quantity is the {\em critical probability for percolation}, defined as:

$$ p_c(\lambda) := \inf\left\{ p : \Pee_{p,\lambda}(\text{there exists an unbounded cluster})>0\right\}. $$

\noindent 
In the case of Poisson-Voronoi percolation on Euclidean space $\eR^d$, elementary considerations
show that the critical probability does not depend on $\lambda$. 
But, for Poisson-Voronoi percolation defined on other ambient geometric spaces, this is typically not the case.
In fact, as shown by Benjamini and Schramm~\cite{BS2001}, on the hyperbolic plane $p_c(\lambda)$ is strictly positive and 
tends to zero as $\lambda$ tends to zero. 
This same behaviour should also occur in $\Haa^d$ with $d\geq 3$, although it does not seem to have been worked out 
explicitly in the literature.\footnote{%
We believe that some of the arguments in the paper~\cite{HansenMuller2024} by
Hansen and the second author can be adapted to give that $p_c(\lambda) = O(\lambda)$ as $\lambda\searrow 0$ 
for Poisson-Voronoi percolation on $\Haa^d$ with any $d\geq 2$.
}

A more spectacular difference between Poisson-Voronoi percolation on Euclidean and hyperbolic space is the 
behaviour of the number of distinct unbounded clusters.
By well-known, classical arguments in percolation theory, Poisson-Voronoi percolation on $\eR^d$ has 
at most one unbounded cluster (almost surely, for any choice of $p,\lambda$).
In contrast, for Poisson-Voronoi percolation on some other ambient geometric spaces, including $\Haa^2$, there are 
choices of the parameters $p,\lambda$ for which there are {\em infinitely many, distinct, unbounded clusters} 
(almost surely). 
A relevant quantity is the {\em critical value for uniqueness}, or {\em uniqueness threshold}, defined by:

$$ p_u(\lambda) := \inf\left\{ p : \Pee_{p,\lambda}(\text{there is exactly one unbounded cluster})>0 \right\}. $$

As might be intuitive to some readers, and was formally shown by Benjamini and Schramm~\cite{BS2001}, 
on the hyperbolic plane $\Haa^2$, we have $p_u(\lambda) = 1 - p_c(\lambda)$.
In particular $\lim_{\lambda\searrow 0} p_u(\lambda) = 1$.

In contrast, very recent results of Greb\'{\i}k and Recke~\cite{GrebikRecke} and D'Achille et al.~\cite{dAchilleVanishing}
establish that $\lim_{\lambda\searrow 0} p_u(\lambda) = 0$ for Poisson-Voronoi percolation 
defined over an ambient geometric space from a certain family of geometric spaces that includes 
the Cartesian products $\Haa^{d_1}\times\dots\times\Haa^{d_k}$ with 
$k, d_1,\dots,d_k \geq 2$.
A natural question, posed by Greb\'{\i}k and Recke~\cite{GrebikRecke}, is whether
we also have $\lim_{\lambda\searrow 0} p_u(\lambda) = 0$ for Poisson-Voronoi percolation on $\Haa^d$ with 
$d\geq 3$. We provide a negative answer to this question.

\begin{theorem}\label{thm:main}
For Poisson-Voronoi percolation on $\Haa^d$ with $d\geq 3$, we have $\displaystyle \inf_{\lambda > 0} p_u(\lambda) > 0$.
\end{theorem}

\noindent
We point out that in dimension $d=2$ this is also true: By the results of Benjamini and Schramm~\cite{BS2001},
$p_u(\lambda) = 1 - p_c(\lambda) > 1/2$ for all $\lambda>0$ when $d=2$.
We also mention that a clever, unpublished argument of D'Achille and Curien (personal communication, December 2025)
gives $p_u(\lambda) \leq 1/2$ for all $\lambda>0$ for Poisson-Voronoi percolation on $\Haa^d$ with $d\geq 3$.

\vspace{1ex}

The {\em ideal Poisson-Voronoi tessellation} (IPVT) is a non-trivial, random tessellation of $\Haa^d$ that, in a certain 
sense, corresponds to the $\lambda \searrow 0$ limit of the hyperbolic Poisson-Voronoi tessellation. 
We refer the reader to Section~\ref{sec:ipvt} or~\cite{dAchilleIdeal} for the precise definition and the details of the mentioned convergence statement.
We define the uniqueness threshold $p_u(\text{IPVT})$ for the IPVT in the obvious way. 
(That is, as the infimum of all $p>0$ for which the probability that percolation on the IPVT has exactly one infinite black cluster is positive.)
As a side remark we note that, as shown in~\cite{dAchilleIdeal}, each cell of the IPVT is adjacent to infinitely many other cells almost surely. 
This implies that, for the IPVT, having exactly one infinite black cluster is the same as all black cells belonging to the same, humongous cluster. And of course also $p_c(\text{IPVT}) = 0$.
With some minor additions to our arguments -- and making use of some results by D'Achille et al.~\cite{dAchilleIdeal} -- we are able to derive the following consequence of our work. 

\begin{corollary}\label{cor:ipvt} 
In dimension $d \geq 3$ we have $p_u(\text{IPVT}) > 0$.
\end{corollary}

\noindent
This answers a question of Curien (personal communication, July 2026). While finalising this paper we learned that Dembin, Gr\'ebik, Recke and Severo have independently derived Corollary~\ref{cor:ipvt} making use 
of the results in an earlier version of this manuscript (personal communication, September 2026). 

\vspace{1ex}

{\bf Related work.} For context, we give a brief and necessarily incomplete overview of related
works.
Percolation theory is a vast and very active area of modern probability theory.
For an overview, we refer the reader to the monographs~\cite{BollobasRiordanboek,Grimmettboek}. 
Poisson-Voronoi tessellations (mostly in $d$-dimensional, Euclidean space) are one of the central models
studied in stochastic geometry, with a considerable history. See the monographs~\cite{SchneiderWeil,StoyanKendallMecke87} and the 
references therein.
While they have typically been studied in the Euclidean space $\eR^d$, by now there is a substantial number of works on 
Poisson-Voronoi tessellations over more general 
geometric spaces, including $\Haa^d$. See for example~\cite{Elliot2, Elliot, BudzinskiEtalCheeger, 
CalkaChapronEnriquez, dAchilleIdeal, dAchilleThaele, MellickEtal, Isokawa3d, Isokawa00, MellickIndistinguish}.
Recently, Poisson-Voronoi tessellations on hyperbolic $d$-space $\Haa^d$ have attracted 
considerable (and renewed) interest. There is not only a lot of activity in the stochastic geometry community, but 
Poisson-Voronoi tessellations are also used as a tool to prove results in other branches of mathematics
(see for example~\cite{BudzinskiEtalCheeger, MellickEtal}).

Percolation on Poisson-Voronoi tessellations was first studied by 
Zvavitch~\cite{Zvavitch} who proved that $p_c \geq 1/2$ in the Euclidean plane.
It took about a decade until Bollob\'as and Riordan~\cite{bollobas2006critical} were able to prove a matching upper bound, 
establishing that $p_c=1/2$ for Poisson-Voronoi percolation on $\eR^2$.
Several works have since contributed to a more detailed picture of Poisson-Voronoi percolation on the Euclidean plane. 
See for example~\cite{AhlbergBaldasso18,AhlbergEtAl16,Tassion16,Vanneuville19}.
For Poisson-Voronoi percolation on $\eR^d$ with $d$ arbitrary, upper and lower bounds on $p_c$ in terms of $d$ 
were given in~\cite{BalisterBollobas10,BalisterBollobasQuas05}
and in~\cite{Duminil19} it was shown that for $p<p_c$ the probability that the cluster of the origin 
has radius at least $n$ decays exponentially in $n$ while the probability that the origin is in 
an infinite cluster is at least $\text{const}\cdot (p-p_c)$ for $p>p_c$.

Poisson-Voronoi percolation on more general two and three dimensional manifolds 
was studied by Benjamini and Schramm in~\cite{BSconformal}, and 
Poisson-Voronoi percolation on the projective plane was studied by Freedman~\cite{Freedman97}.
Poisson-Voronoi percolation on $\Haa^2$ was first studied by Benjamini and Schramm 
in the influential paper~\cite{BS2001}. 
Their results include that $0<p_c(\lambda)<1/2$ and $p_u(\lambda)=1-p_c(\lambda)$ for all $\lambda>0$, that 
$\lambda \mapsto p_c(\lambda)$ is continuous, that $\lim_{\lambda\searrow 0} p_c(\lambda)=0$ and that 
for $p_c(\lambda)<p<p_u(\lambda)$ there are infinitely many unbounded clusters almost surely.
Hansen and the second author proved in~\cite{HansenMuller2022} that, for Poisson-Voronoi percolation 
on the hyperbolic plane, $p_c(\lambda)\to 1/2$ as $\lambda\to\infty$, 
and in~\cite{HansenMuller2024} that $p_c(\lambda) = (\pi/3)\cdot\lambda + o(\lambda)$ as $\lambda\searrow 0$.
This settled a conjecture, respectively an open problem, posed by Benjamini and Schramm~\cite{BS2001}.

A preprint of Li and Liu~\cite{LiYu} proves an analogue for $\Haa^d$ of the results for $\eR^d$ in~\cite{Duminil19}.
Recent preprints by Greb\'{\i}k and Recke~\cite{GrebikRecke} and D'Achille et al.~\cite{dAchilleVanishing}
show that $\lim_{\lambda\searrow 0}p_u(\lambda)=0$ for Poisson-Voronoi percolation defined on an ambient space
from a family of geometric spaces that includes Cartesian products $\Haa^{d_1}\times\dots\times\Haa^{d_k}$ with 
$k, d_1,\dots,d_k \geq 2$.
A recent preprint of B\"uhler et al.~\cite{BDRS25} states that, for Poisson-Voronoi percolation 
on an ambient geometric space from a certain family of geometric spaces that includes $\Haa^d$, we have 
$\lim_{\lambda\to\infty} p_u(\lambda) = \lim_{\lambda\to\infty} p_c(\lambda) = p_c(\eR^d)$, where
$p_c(\eR^d)$ denotes the critical probability for Poisson-Voronoi percolation on $\eR^d$.
This proves and extends a conjecture of Hansen and the second author~\cite{HansenMuller2024}.

In contrast to ordinary Poisson-Voronoi tessellations, the {\em ideal Poisson-Voronoi tessellation} (IPVT) has only been introduced very recently.
It occurred implicitly in~\cite{BhupatirajuThesis} and~\cite{BudzinskiEtalCheeger}, but was first formally constructed in 
the pioneering work~\cite{dAchilleIdeal}. The IPVT has since caught the attention of the community. 
Works studying its properties include~\cite{dAchilleThaele, MellickIndistinguish}.
To the best of our knowledge, percolation on the IPVT has not been considered before.

\subsection{The structure of the paper.}

In the next section, we provide an informal overview of some of the ideas in our paper, to aid the reader 
in digesting it.
In Section~\ref{sec:prelim}, we list some notations, definitions and results from the literature that we will 
rely on in the rest of the paper.
Section~\ref{sec:setup} gives some more preparatory observations for the proofs. In Section~\ref{sec:edgeprob}
we provide an upper bound on the probability of two Voronoi cells being adjacent.
In Section~\ref{sec:gracarmodel}, we consider an auxiliary continuum percolation model on $\Haa^d$ and 
in particular we prove an upper bound on the expected number of a certain kind of paths in that model.
In Section~\ref{sec:redpath} we show that the bounds from the previous section also apply to 
a certain kind of substructures of the Poisson-Voronoi tiling that we call {\em reduced paths}.
In Section~\ref{sec:wrapitup} we combine the results from the previous sections
into a proof of Theorem~\ref{thm:main}.
The derivation of Corollary~\ref{cor:ipvt} can be found in Section~\ref{sec:ipvt}.

In order not to disrupt the flow of the main 
arguments while at the same time keeping our paper as ``complete'' and self-contained as possible, 
we have placed a number of proofs in the Appendix.

\subsection{Informal sketch of the proofs and the intuitions guiding them.}

The overall plan for the proof of Theorem~\ref{thm:main} is to show that, when $p$ is a small constant 
dependent only on the dimension then, for all $\lambda>0$, we have 
$\Pee_{p,\lambda}(x \pijl y) \to 0$ as $\dist(x,y)$ tends to infinity. 
A well-known observation (stated as Proposition~\ref{prop:pu} below) tells us that this implies $p \leq p_u(\lambda)$.
Let us also mention that we can restrict ourselves to $\lambda$ less than some arbitrary positive constant
by relatively standard considerations involving $p_c(\lambda)$ (Proposition~\ref{prop:pcpos} below).

The {\em Poisson-Delaunay graph} is the graph whose vertex set is $\Zcal$ and where $z,w \in \Zcal$ are connected by an edge 
if the Voronoi cells $C(z), C(w)$ intersect. 
So Poisson-Voronoi percolation is the same as site percolation 
on the Delaunay graph. 

In this sketch as well as in most of the rest of the paper, it is convenient to work in the 
Poincar\'e halfspace model of $\Haa^d$. (See Section~\ref{sec:hyperbolicgeom} for the precise
definition and some relevant facts.) 
Quick ``back of the envelope'' calculations suggest that, at least for small $\lambda$, the (expected) 
number of neighbours in the Delaunay graph of a given point $z \in \Zcal$ grows like 
$\text{const} \cdot \lambda^{-1}$, but only constantly many neighbours lie above $z$ (in expectation).
Here and in the rest of the paper ``above'' refers to the value of the $d$-th coordinate in the halfspace model.
(We do not provide the mentioned back of the envelope calculations anywhere in the paper. We also do 
not flesh them out into more formal statements, since we will not rely on these observations
in any of the proofs.)

This suggests that for $p \geq c \cdot \lambda$ with $c$ a sufficiently large constant, we can already 
expect percolation to occur, but a typical infinite path travels ``downwards'' towards
$\partial\Haa^d = \eR^{d-1}\times\{0\}$.
On the other hand, even for much larger $p$, i.e.~when $p$ is at most a small constant, it should be hard to 
travel ``upwards''.
This in turn suggests that if $z,w$ are far apart but at roughly the same height, and $p$ is at most a small constant, 
it is unlikely that there is a path between $z$ and $w$.
Further credence for this idea is given by the observation that the geodesic (shortest curve) between $z,w$ is a circle segment 
that lies above $z,w$ and reaches much higher than $z,w$ if they are far apart and at roughly the same height. (See again Section~\ref{sec:hyperbolicgeom}
for precise statements etc.) 
It is of course possible that $z,w$ are connected by a black path that does not stay close to this geodesic, but 
this will make the path (much) longer in terms of the hyperbolic length, and presumably also the number 
of vertices on it should be much larger, etc.

One of the issues that makes Poisson-Voronoi percolation challenging is ``dependencies''. 
The events that $z,w$, respectively $z',w'$, are adjacent in the Delaunay graph are dependent, no matter
the positions of $z,w,z',w'$ in $\Haa^d$. (To make this from a vague into a more formal statement 
we can consider conditioning on $z,w,z',w' \in \Zcal$ in an appropriate sense.)

One of the first things that comes to mind when trying to decide what kind of behaviour to expect
is to first consider a simplified model, where the adjacencies are independent.
That is, to get an idea what to expect we consider a graph with vertex set $\Zcal$ and 
for each pair $z,w \in \Zcal$ we include the edge $zw$ with probability $P(z,w)$, independent of 
all other pairs. 
Here we could take $P$ to be the probability that $z,w$ are adjacent in the Delaunay graph, which is a
decaying function of $\dist(z,w)$. 
But, with foresight, instead we choose a $P$ that decays more slowly in $\dist(x,y)$, 
to create some wiggle room when we transfer the results back to the original model.

We are still just in the stage of fixing ideas, so for the moment let us simplify our life even further 
and just discard all points of $\Zcal$ below some height, say height 1.
That is, we have a graph with vertex set $\Zcal \cap (\eR^{d-1}\times [1,\infty))$ and edges 
are included independently according to $P(.,.)$.
If we project the points of $\Zcal \cap (\eR^{d-1}\times [1,\infty))$ onto the first $d-1$ coordinates then we 
obtain a homogeneous Poisson point process
$\Xcal$ on $\eR^{d-1}$ of intensity $\lambda/(d-1)$. We can think of the $d$-th coordinates as 
(independent) weights attached to the points of this Poisson point process, 
satisfying $\Pee( W > w ) = w^{-(d-1)}$ for $w\geq 1$.
The probability of two points $x, y \in \Xcal$ being connected now is a function of their weights $W_x$ and $W_y$ and the Euclidean 
distance $\norm{x-y}$. 

The model we've arrived at fits under the umbrella of {\em weighted random connection models}. 
Various sets of authors have introduced similar, but different, such models. For us the rather general model of Gracar et 
al.~\cite{GracarEtal} is convenient. 
For $P$ an appropriate function of $\norm{x-y}, W_x, W_y$ (see~\eqref{eq:Pdef} below for the 
precise definition), we have that $P(z,w)$ dominates the probability that $zw$ is an edge of the Poisson-Delaunay graph, and 
our model fits the framework of Gracar et al.
Specifically, the version of their model with the ``preferential attachment kernel'' and parameter
$\gamma=1/2$. 
(This choice of parameter happens to be a delicate ``border case'' where the second moment of the number of neighbours of 
a typical point switches from being finite to infinite. It corresponds to power-law exponent $\tau=3$ in language that 
is common in the complex networks community.)
Theorem 1.1 in~\cite{GracarEtal} tells us that in this version of the weight-dependent random connection model 
there is a non-trivial percolation threshold, that is independent of
$\lambda$. 

It does however not seem possible to use this result ``off the shelf''.
We wish to show that the probability that $z \pijl w$ is small when $\dist(z,w)$ is large. 
Even if a path between $z, w \in \eR^{d-1} \times [1,\infty)$ is
unlikely in the model where we discard $\Zcal \cap \left( \eR^{d-1}\times(0,1) \right)$, it is not obvious
that keeping these points cannot somehow make paths between $z,w$ (much) more likely.

To avoid this issue, we just revert back to the model defined on $\eR^{d-1}\times(0,\infty)$ with the same choice of $P$.
Luckily some relevant intermediate results of Gracar et al.~still hold true in this setting.
We consider the expected number $S_n(h)$ of paths between the reference
point $o=(0,\dots,0,1)$ and some (black) point in $\Zcal \cap \left( \eR^{d-1}\times [h,\infty) \right)$, that 
are ``shortcut free''. (Our choice of $P$ is such that $P(z,w)=1$ if the distance between
$z,w$ is small. A shortcut is an edge between non-consecutive points of the path that is included
automatically.)
Proposition~\ref{prop:Snh} below establishes a bound on $S_n(h)$ that works for all $0 < \lambda, p \leq 1$ and $h > 1$.
Among other things, its proof uses the Mecke formula (repeated as Theorem~\ref{thm:Mecke} below) 
and Lemmas 2.2 and 2.3 from~\cite{GracarEtal}.
The reason for considering $S_n(h)$ and not, for instance, just the expected number of paths of length $n$ between two 
far away vertices $z,w$ is that this latter expectation tends to infinity with $n$ 
(by calculations we again do not provide in the paper).

We now wish to leverage the bound on $S_n(h)$ in the ``independent edge model'' to 
show that the probability that there is a black path in the Poisson-Delaunay graph 
between $o$ and a black Poisson point in $\eR^{d-1}\times[h,\infty)$ is small.
If we compute the expected number of paths between $z$ and $w$ of length $n$ using the Mecke formula then we obtain an integral over $z_1,\dots,z_{n-1}$ wrt.~the hyperbolic volume measure 
(see Section~\ref{sec:hyperbolicgeom} below for the precise definition), whose integrand is the 
probability $\Pee(\text{$z_{i-1}z_i$ is an edge for $i=1,\dots,n$})$ 
(where for convenience we set $z_0:=z, z_n:=w$).
Because of dependencies this probability can be much larger than the product of the probabilities 
$\prod_{i=1}^n \Pee(\text{$z_{i-1}z_i$ is an edge})$ when $z_0,\dots,z_n$ are positioned unfavourably. 

As a first step towards resolving this issue, we use a more generous notion of adjacency that 
reduces geometric dependencies. 
The {\em Gabriel ball} $\BGab(z,w)$ of $z,w$ is the ball whose center is the 
midpoint of the geodesic between $z,w$ and whose radius is $\dist(z,w)/2$. We let $E(z,w)$ denote the 
event that the {\em white} points inside $\BGab(z,w)$ are consistent with $zw$ being an edge
of the Poisson-Delaunay graph. 
The probability of $E(z,w)$ decays rather fast with the distance $\dist(z,w)$, as shown by Lemma~\ref{lem:edgeprob} below.
As alluded to earlier, we've ensured that the function $P(z,w)$ that we've used in the ``independent edge model''
decays much slower than $\Pee( E(z,w) )$.
Note that $E(z,w), E(z',w')$ are independent whenever the corresponding Gabriel balls do not intersect. 
This does not make all the issues with dependencies go away of course. 
The probability that $z_0,\dots,z_n$ is a path can still be much larger than the product of the 
probabilities $\prod_{i=1}^n \Pee( E(z_{i-1},z_i) )$, if there 
is a lot of overlap between the Gabriel balls.

As the second step towards dealing with dependencies, we switch attention to what we call {\em reduced paths}.
Given a path $z_0,\dots,z_n$ in the Poisson-Delaunay graph, 
we create its reduced path iteratively as follows. Starting from $i=0$, if 
there is a shortcut $z_iz_{i+j}$, i.e.~$P(z_i,z_{i+j}) = 1$ and $j>1$, then 
we omit $z_{i+1},\dots,z_{i+j-1}$ from the path and start the next iteration at $z_{i+j}$. 
And, if $\BGab(z_{i-1},z_i)$ has a lot of overlap with the Gabriel balls of 
previous edges on the path, then we allow ourselves to skip forward to a vertex $z_{i+j}$ that is within a certain
distance of $z_i$, where the distance depends on $\dist(z_{i-1},z_i)$ and the amount of overlap.
(In order to give the reader a gist of the arguments without burdening them too much with details, we've 
been a bit imprecise here. See Section~\ref{sec:redpath} for the full definition and arguments.)
We now let $R_n(h)$ be the analogue for reduced paths of $S_n(h)$.
That is, $R_n(h)$ is the expected number of reduced paths between $o$ and some black point in 
$\Zcal \cap \left(\eR^{d-1}\times[h,\infty)\right)$ such that all but the last point have height $<h$.
Proposition~\ref{prop:Rh} below shows that $R_n(h) \leq S_n(h)$ for all $p\leq 1/2$ and 
$0<\lambda \leq 1$.
For any reduced path $z_0,\dots,z_k$, some of its consecutive pairs of vertices were also consecutive on the 
original path, and in particular the corresponding events $E(.,.)$ must hold.
The main step in the proof of Proposition~\ref{prop:Rh} is to show that the probability of 
$E(.,.)$ holding for these remaining original edges is no more than the product $\prod_{i=1}^k P(z_{i-1},z_i)$
over all edges of the reduced path.
Since a reduced path is short-cut free by construction, the bound $R_n(h) \leq S_n(h)$ then follows
(invoking the Mecke formula once more).

In order to achieve this main step in the proof, having chosen the definition of reduced paths carefully, we 
exploit the fact that $\Pee(E(z,w)) \ll P(z,w)$ as well as the 
curious behaviour of overlapping balls in hyperbolic geometry, to ``redistribute the costs'' of the 
remaining original edges over all the edges of the reduced path. 
This part of the paper is probably the most technical, but also the most inventive. We consider it the most novel mathematical contribution of our paper. 

Having established that $R_n(h) \leq S_n(h)$, our upper bound on $S_n(h)$ shows that, 
for $p \leq p_0$ and $0\leq \lambda\leq 1$ where $p_0>0$ is a constant dependent only on the dimension, 
we have $\sum_n R_n(h) \to 0$ as $h\to\infty$.
Fairly elementary additional arguments then allow us to 
show that $\Pee( o \pijl (0,\dots,0,h) ) \to 0$ as $h\to\infty$, for all $p\leq p_0$ and 
$0\leq \lambda\leq 1$. This implies $p_u(\lambda) \geq p_0$ for all $0 < \lambda \leq 1$ by earlier remarks.
As mentioned before, a constant, positive lower bound on $p_u(\lambda)$ for $\lambda \geq 1$ follows by 
fairly standard considerations involving $p_c(\lambda)$.

\section{Notation and preliminaries.\label{sec:prelim}}

We will use $\Bi(n,p)$ to denote the binomial distribution with parameters $n$ (the number of coins) and $p$ (the success 
probability) and $\Po(\nu)$ will denote the Poisson distribution with parameter (mean) $\nu$.
We'll make use of the following Chernoff-type bound.
A proof can for instance be found in~\cite{Penroseboek} (Lemmas 1.1 and 1.2). 

\begin{lemma}[Chernoff bound]\label{cher}
If $X$ is either Poisson or binomially distributed it holds for all $0\leq x\leq \Ee X$  
$$ \Pee(X\leq x)\leq e^{-H(x/\Ee X)\cdot\Ee X}, $$

\noindent
where $H(a)=1-a+a\ln(a) $ and $H(0)=1$.
\end{lemma}

\subsection{Some (random) Euclidean geometry.}

We denote by $S^{d-1} \subseteq \eR^d$ the $(d-1)$-dimensional unit sphere (in ambient $d$-dimensional space).
We will use $\kappa_d$ to denote the $d$-dimensional volume of the $d$-dimensional unit ball 
(wrt.~the plain vanilla metric and volume in the ambient Euclidean space $\eR^d$); and we will use 
$\omega_{d-1}$ to denote the ``$(d-1)$-dimensional area'' of $S^{d-1}$. That is,

\begin{equation}\label{eq:volballEucl} 
\kappa_d := \frac{\pi^{d/2}}{\Gamma(d/2+1)}, \quad 
\omega_{d-1} := \frac{2 \pi^{d/2}}{\Gamma(d/2)}, 
\end{equation}

\noindent
where $\Gamma(.)$ denotes the gamma function. For a proof that these constants indeed 
correspond to the stated quantities, see for example Corollaries 15.14 and 15.15 in~\cite{Schilling}.

The {\em (positive) cone} of opening $\alpha$ around the vector $v \in \eR^d \setminus \{\orig\}$ is defined as

$$ \kone(v,\alpha) := \{ w \in \eR^d : \angle v\orig w < \alpha \}, $$

\noindent
A {\em linear halfspace} is a set of the form 

$$ L := \{ x \in \eR^d : v^t x > 0 \} = \kone(v,\pi/2), $$

\noindent 
with $v \in \eR^d \setminus \{\orig\}$.

We'll say a random vector $X$ on $\eR^d$ is {\em rotationally invariant} if $X \isd A(X)$ for every 
orthogonal map $A : \eR^d \to \eR^d$.

We'll make use of the following well-known fact.

\begin{lemma}\label{lem:kap} 
There exists a constant $c=c(d)$ such that 
for any $v \in \eR^d \setminus \{\orig\}$ and $0 \leq \alpha \leq \pi$, we have

$$  \Pee( X \in \kone(v,\alpha) ) \leq c \cdot \alpha^{d-1}, $$

\noindent 
for every atom-free, rotationally invariant $X$.
\end{lemma}

For a proof, see for instance Lemma 2.1 in~\cite{BriedenEtal}, which is a more detailed result that 
implies the bound we need.

Another well-known result we'll make use of is Wendel's formula.

\begin{lemma}[Wendel's formula~\cite{Wendel62}]\label{lem:wendel}
Let $X_1,\dots,X_n$ be i.i.d.~from some atom-free, rotationally invariant distribution. Then

$$ \Pee\left(\begin{array}{l}
\text{there is some linear half-space $L$}\\
\text{such that $X_1,\dots, X_n \in L$}\end{array}\right)
=\Pee(\emph{\text{Bi}}(n-1,1/2)\leq d-1 ). $$

\end{lemma}

\noindent
(Wendel's result is stated in~\cite{Wendel62} for the uniform distribution on $S^{d-1}$. 
The version for general rotationally invariant and atom-free distributions follows
by considering the renormalized vectors $X_i/\norm{X_i}$.)

\subsection{Hyperbolic $d$-space.\label{sec:hyperbolicgeom}}

The $d$-dimensional hyperbolic space is a natural generalization of the hyperbolic plane to general dimension $d$.
It can be defined as the (unique up to isometry) simply connected, $d$-dimensional Riemannian manifold of constant
sectional curvature $-1$. (For the definitions of the terms involved, see for example~\cite{Ratcliffe}. Their meaning will not play
any role in the rest of the paper.)
Hyperbolic $d$-space has several models (``coordinate maps''), including the hyperboloid model, the Poincar\'e halfspace model, 
the Poincar\'e ball model and the Klein ball model, all defined analogously to the eponymous models of the hyperbolic plane.
See for example~\cite{CannonFloydKenyonParry1997} or~\cite{Ratcliffe} for a thorough treatment. 

We shall be working exclusively in the Poincar\'e halfspace model and the Poincar\'e ball model of hyperbolic $d$-space.
In fact, we will mostly use the Poincar\'e halfspace model, and only occasionally switch to 
the Poincar\'e ball model when it is convenient for the argument in question. 
We briefly introduce these two models and list the properties and facts we shall need in the rest of the paper.

In the Poincar\'e halfspace model, we equip the open halfspace $\Haa^d := \eR^{d-1} \times (0,\infty)$ with the metric
defined by stipulating that the hyperbolic length of a (differentiable) curve $f : [0,1] \to \Haa^d$
is given by 

\begin{equation}\label{eq:metricdef} 
\text{length}_{\Haa^d}(f) = \int_0^1 \frac{\norm{f'(t)}}{f_d(t)}{\dd}t. 
\end{equation}

Naturally the {\em hyperbolic distance} $\dist(z,w)$ between $z,w \in \Haa^d$ is defined as the 
length of the shortest curve connecting $z,w$.
It can be written explicitly as

\begin{equation}\label{eq:distdef} 
\dist(z,w) = 2 \arcsinh\left(\frac{\norm{z-w}}{2\sqrt{z_dw_d}}\right). 
 \end{equation}

\noindent
Theorem~4.6.1 in~\cite{Ratcliffe} gives the equivalent identity
\begin{equation*}
 \cosh(\operatorname{dist}(z,w))
 =
 1+\frac{\|z-w\|^2}{2z_dw_d}.
\end{equation*}
Formula~(3) follows from
$\cosh(2\operatorname{arcsinh}t)=1+2t^2$ (using the double-angle formula) and the injectivity of
$\cosh$ on $[0,\infty)$.
The curve of shortest hyperbolic length between two points $z, w \in \Haa^d$ is called a {\em geodesic}.
Geodesics satisfy the following visually pleasing property.

\begin{equation}\label{eq:factgeodesicH}
\begin{array}{l}
\text{The geodesic between $z, w \in \Haa^d$ is a line segment if $z,w$ differ only in the $d$-th}\\
\text{coordinate, and otherwise it is the segment of a circle through $z,w$ that hits the}\\
\text{bounding hyperplane $\partial\Haa^d = \eR^{d-1}\times\{0\}$ at right angles.}
\end{array}
\end{equation}

\noindent
For a proof, see for example~\cite[Theorem 9.3]{CannonFloydKenyonParry1997}.
Throughout the paper, we fix the natural reference point $o := (0,\dots,0,1) $ in the halfspace model. We may sometimes refer to $o$ as ``the origin''.
We denote the ball wrt.~$\dist(.,.)$ around $x \in \Haa^d$ of radius $r > 0$ by:

$$ B(x,r) := \left\{ y \in \Haa^d : \dist(x,y) < r \right\}. $$

\noindent 
In a few places in the paper, we will want to also consider balls with respect to the Euclidean metric. 
To avoid any possible confusion, we will always denote balls with respect to the Euclidean metric by: 

$$ B_{\eR^d}(x,r) := \left\{ y \in \eR^d : \norm{x-y}<r \right\}. $$

\noindent
For us an important observation is:

\begin{equation}\label{eq:fact} 
\begin{array}{l} 
\text{Every ball wrt.~$\dist$ is also a Euclidean ball, and every Euclidean ball} \\
\text{contained in $\Haa^d$ and not tangent to $\partial\Haa^d$ is also a ball wrt.~$\dist$.}
\end{array} 
\end{equation}

\noindent 
(But the centers and radii wrt.~the different metrics do not coincide.)
For a proof see for instance Fact 1 in~\cite{CannonFloydKenyonParry1997} or Theorem 4.6.4 
in~\cite{Ratcliffe}.

If $a,b,c \in \Haa^d$ are distinct, and $\gamma$ denotes the angle 
that the geodesic line segment between $a,c$ and the geodesic line segment between $b,c$ make at their intersection point $c$, the 
{\em hyperbolic cosine rule} states that:

\begin{equation}\label{eq:hypercosrule} 
\begin{array}{rcl} 
\cosh(\dist(a,b)) & = & \cosh(\dist(a,c))\cdot \cosh(\dist(b,c)) \\
 & & - \cos(\gamma)\cdot\sinh(\dist(a,c))\cdot\sinh(\dist(b,c)).
\end{array}
\end{equation}

\noindent 
For a proof see~Theorem 3.5.3 in~\cite{Ratcliffe}. 
(The proof in~\cite{Ratcliffe} is phrased for dimension $d=2$ only. In dimension $d\geq 3$
we can apply an isometry to achieve the situation where $a,b,c \in \eR \times \{0\}^{d-2}\times (0,\infty)$, 
a subspace of $\Haa^d$ which is isometric to the hyperbolic plane $\Haa^2$. The earlier 
remarks on geodesics imply that the geodesics involved will stay inside this subspace.)

The Poincar\'e halfspace model is equipped with a natural notion of volume that is implied 
by the choice of metric. (See Theorem 4.6.7 in \cite{Ratcliffe}.)
The hyperbolic volume of a (Borel) measurable set $A \subseteq \Haa^d$ is given by:

$$ \mu(A) = \int_A (1/z_d)^d{\dd}z, $$

\noindent
For $f : A \to \eR$ a (Borel) measurable function, the integral 
with respect to the hyperbolic volume measure is given by:

$$ \int_A f(z) \mu({\dd}z) := \int_A (f(z)/z_d^d){\dd}z. $$

\noindent 
The hyperbolic volume of a ball of radius $r$ equals:

\begin{equation}\label{eq:volball} 
\mu\left( B(x,r) \right) = \omega_{d-1} \cdot \int_0^r \left(\sinh s\right)^{d-1} {\dd}s. 
\end{equation}

\noindent
(See for example~\cite{Ratcliffe}, Exercise 3.4(6).)
For convenient future reference we mention the following straightforward consequences of~\eqref{eq:volball}.

\begin{equation}\label{eq:volballUB}
 \mu\left( B(x,r) \right) \leq \frac{\omega_{d-1}}{2^{d-1}(d-1)} \cdot e^{(d-1)r},
\end{equation}

\begin{equation}\label{eq:asympvolball}
 \mu\left( B(x,r) \right) = (1+o_r(1)) \cdot \frac{\omega_{d-1}}{2^{d-1}(d-1)} \cdot e^{(d-1)r}. 
\end{equation}

\begin{equation}\label{eq:volballrminuss}
\mu( B(x,r-s) ) \leq e^{-(d-1)s} \cdot \mu( B(x,r) ).
\end{equation}

\noindent 
For completeness we provide derivations of~\eqref{eq:volballUB},~\eqref{eq:asympvolball} and~\eqref{eq:volballrminuss} in
Appendix~\ref{sec:ballrminuss}. 

\vspace{1ex}

The {\em Poincar\'e ball model} is the $d$-dimensional unit ball $\Bee^d := B_{\eR^d}(\orig,1)$
equipped with a certain metric and volume measure.
One way to define these is to declare that the $d$-dimensional Cayley transform\footnote{The $d$-dimensional
Cayley transform $\varphi : \Haa^d \to \Bee^d$
is given by $\varphi_i(x) = 2x_i/(\norm{x}^2+2x_d+1)$ for $1\leq i \leq d-1$ and
$\varphi_d(x) = (\norm{x}^2-1)/(\norm{x}^2+2x_d+1)$.
The two dimensional version can be identified with the M\"obius 
transform $z \mapsto i\cdot(z-i)/(z+i)$.} 
$\varphi : \Haa^d \to \Bee^d$ be an isometry, and let the metric and volume measure carry over from the Poincar\'e halfspace model.
That is, for $z, w\in \Bee^d$, respectively $A \subseteq \Bee^d$ a (Borel) measurable set, 
we define

$$ \text{dist}_{\Bee^d}(z,w) := \dist(\varphi^{-1}(z),\varphi^{-1}(w) ), \quad  
\mu_{\Bee^d}(A) := \mu( \varphi^{-1}(A) ). $$

\noindent
Via a standard computation (involving the Jacobian of $\varphi$) this gives:

\begin{equation}\label{eq:Poincareballvolmeas} 
\mu_{\Bee^d}(A) = \int_A \frac{2^d}{(1-\norm{z}^2)^d}{\dd}z. 
\end{equation}

\noindent
(An alternative derivation is given in~\cite{Ratcliffe}, Theorem 4.5.6.)
The angle between two $\Haa^d$-geodesics $\gamma,\gamma'$ 
meeting in a common point $p \in \Haa^d$ is the same as the 
angle that the corresponding $\Bee^d$-geodesics $\varphi[\gamma], \varphi[\gamma']$ make in $\varphi(p) \in \Bee^d$. This identity can also be found in \cite[Section 4.6, paragraph following Corollary~3]{Ratcliffe}.
In particular the hyperbolic cosine rule also holds in $\Bee^d$.
What is more, analogously to~\eqref{eq:factgeodesicH} above:%
\begin{equation}\label{eq:factball}
\begin{array}{l}
\text{The geodesic between $z,w \in \Bee^d$ is a straight line segment if $z,w$ lie on a line}\\
\text{through the origin $\orig$ and otherwise it is the segment of a circle through $z,w$}\\
\text{that meets $\partial\Bee^d = S^{d-1}$ at right angles.}
\end{array}
\end{equation}

\noindent
(For a proof, see for example Corollary 2 on Page 123 of~\cite{Ratcliffe}.)
And, analogously to~\eqref{eq:fact}:

\begin{equation}\label{eq:factball2}
\begin{array}{l}
\text{Any ball wrt.~$\dist_{\Bee^d}$ is also a ball wrt.~the Euclidean metric and any Euclidean}\\ 
\text{ball contained in $\Bee^d$ and not tangent to $S^{d-1}$ is a ball wrt.~$\dist_{\Bee^d}$.}
\end{array}
\end{equation}

\noindent
For a proof see for example Fact 1 in~\cite{CannonFloydKenyonParry1997} or Theorem 4.5.4 in~\cite{Ratcliffe}.

Hyperbolic $d$-space is {\em homogeneous}, meaning that for every $p, q \in \Haa^d$ there is an isometry
from $\Haa^d$ to itself that maps $p$ to $q$. The same is of course true in $\Bee^d$.
(For a proof see for example Corollary 4 in Section 3.2 of \cite{Ratcliffe}.)
Moreover, all isometries from $\Haa^d$ to itself, respectively $\Bee^d$ to itself, preserve hyperbolic volume. 
(For a proof, see for example the paragraph following formula (3.4.6) in~\cite{Ratcliffe}.)

\subsection{Hyperbolic Poisson point processes.\label{sec:hyperPPP}}

Throughout this paper $\Zcal$ denotes a homogeneous Poisson point process (PPP) on $d$-dimensional hyperbolic space. 
Analogously to homogeneous Poisson point processes on $\eR^d$, a homogeneous Poisson process $\Zcal$ of intensity 
$\lambda$ on $\Haa^d$ is characterized completely by the properties that
{\bf a)} for each (Borel measurable) set $A\subseteq \Haa^d$ with $\mu(A)<\infty$ the random variable $|\Zcal \cap A|$ is Poisson distributed with mean 
$\lambda \cdot \mu(A)$, and {\bf b)} if $A_1, \dots, A_m \subseteq \Haa^d$ are (Borel measurable and) disjoint then the 
random variables $|\Zcal\cap A_1|, \dots, |\Zcal\cap A_m|$ are independent.
In particular, we can just view $\Zcal$ as an {\em inhomogeneous} Poisson point 
process on the $\eR^d$ with intensity function $u \mapsto \lambda \cdot 1_{\Haa^d}(u) \cdot u_d^{-d}$.
Or, equivalently, in case we work in the Poincar\'e ball model then 
we have an inhomogeneous Poisson point process on $\eR^d$ of intensity
$u \mapsto \lambda \cdot 1_{\Bee^d}(u) \cdot 2^d / (1-\norm{u}^2)^d$.
(This latter PPP $\Ycal$ on $\Bee^d$ we can obtain from the former PPP $\Zcal$ on $\Haa^d$ as $\Ycal = \varphi[\Zcal]$  
where $\varphi$ is the canonical -- or any other -- 
isometry $\varphi : \Haa^d \to \Bee^d$, 
by the mapping theorem for Poisson processes. See for example~\cite{Kingman}, page 18.) 

Throughout the paper, we attach to each point of $\Zcal$ a randomly and independently chosen
colour. (Black with probability $p$ and white with probability $1-p$.)
We let $\Zcal_b$ denote the black points and $\Zcal_w$ the white points of $\Zcal$.
By the marking and thinning theorems for Poisson point processes (see for example~\cite{LastPenrose2017}, Theorem 5.6 and Corollary 5.9),
$\Zcal_b, \Zcal_w$ are {\em independent} hyperbolic 
homogeneous Poisson point processes with intensities $\lambda \cdot p$, respectively $\lambda \cdot (1-p)$.
To emphasize the distinction between the points of $\Zcal$ and other points of $\Haa^d$ we will sometimes
refer to the points of $\Zcal$ as the {\em nuclei} (of the Voronoi cells).

We will rely heavily on a specific case of the Mecke formula. The Mecke formula is sometimes also called 
Campbell-Mecke formula or Slivniak-Mecke formula in the literature. 
It is our weapon of choice
for counting $k$-tuples of points $z_1,\dots,z_k \in \Zcal_b$ satisfying some desired property.
Before stating it, we remind the reader that formally speaking a Poisson process on $\Haa^d$ 
is a random variable that takes values in the space $\Omega_{\text{PPP}}$ of locally finite subsets of $\Haa^d$, 
equipped 
with the sigma algebra generated by the family of events of the form: a given Borel set $B$ 
contains precisely $k$ points.

\begin{theorem}[Mecke formula]\label{thm:Mecke}
For $k \in \eN,\lambda>0$ and $0\leq p\leq 1$, let $\Zcal_b,\Zcal_w$ be as above and let
$g : (\Haa^d)^k \times \Omega_{\text{PPP}} \times \Omega_{\text{PPP}} \to [0,\infty)$ be a nonnegative, measurable function.
Then

$$ \begin{array}{c} 
\displaystyle
\Ee\left[\sum_{\substack{z_1,...,z_k\in\Zcal_b\\ \text{distinct}}} g(z_1,...,z_k,\Zcal_b,\Zcal_w)\right] \\
= \\
\displaystyle 
\left(p\lambda\right)^k \int_{\Haa^d}\dots\int_{\Haa^d} \Ee_{p,\lambda}
\left[g(u_1,...,u_k,\Zcal_b \cup \{u_1,\dots,u_k\},\Zcal_w)\right] 
\mu({\dd}u_1)\dots\mu({\dd}u_k).
\end{array} $$

\end{theorem}

\noindent
As mentioned above, the version we state here is a specific case of a more general result.
The general version can for instance be found in~\cite{SchneiderWeil}, as Corollary 3.2.3. 
A derivation of the specific statement we give here can be found in~\cite{HansenMuller2024}, Corollary 5. 
(The result is phrased for dimension $d=2$ in~\cite{HansenMuller2024}, but the proof works in all dimensions.) 

The {\em Palm version} of the Poisson point process $\Zcal$ wrt.~a tuple of points $z_1,\dots, z_k$ is where 
we generate the process as usual, but then (deterministically) include the points $z_1,\dots,z_k$ afterwards.
That is, we consider the random point set $\Zcal \cup \{z_1,\dots,z_k\}$ instead of $\Zcal$.
In this paper we will only ever include black points in our Palm versions. We use the 
notations 
$\Pee_{p,\lambda}^{z_1,\dots,z_k}(.)$, respectively $\Ee_{p,\lambda}^{z_1,\dots,z_k}(.)$, to denote 
the probability measure, respectively expectation, with respect to the 
pair of point processes $(\Zcal_b \cup \{z_1,\dots,z_k\}, \Zcal_w)$.
(So we could have phrased the Mecke formula above using these notations.)

\section{Proofs.\label{sec:proofs}}

\subsection{Setup.\label{sec:setup}}

As mentioned earlier, the {\em Poisson-Delaunay graph} is the graph whose vertex set is 
$\Zcal$ and where $z,w \in \Zcal$ are connected by an edge if the Voronoi cells $C(z), C(w)$ intersect. 
We will mostly take the perspective of Poisson-Voronoi percolation as 
site percolation on the Poisson-Delaunay graph. This often results in more concise formulations, etc.

We remark that if $z,w \in \Zcal$ are distinct then any point $v \in C(z)\cap C(w)$ must satisfy 
$\dist(v,z)=\dist(v,w)$ and the open ball around $v$ of (hyperbolic) radius $\dist(v,w)=\dist(v,z)$ cannot
contain any point of $\Zcal$. This gives rise to the following alternative 
description of adjacency in the Delaunay graph : 

\begin{equation}\label{eq:DelEdge}\begin{array}{l}
\text{$z,w \in \Zcal$ are adjacent if and only if there is some open ball $B$ such that}\\
\text{$z, w\in \partial B$ and $B \cap \Zcal = \emptyset$.}
\end{array}
\end{equation}

\noindent
Note there is a priori no bound on the size of such a {\em Delaunay ball}.
When the positions of $z,w$ and $\Zcal$ are unfortunate, we can be in the situation 
where there is such a Delaunay ball, but the smallest Delaunay ball is huge compared to the distance $\dist(z,w)$. 
For this reason, we find it convenient to work with a more generous notion of adjacency, 
that does not require us to examine the Poisson point process outside a bounded region.
For $z, w \in \Haa^d$, we define their {\em Gabriel ball} $\BGab(z,w)$ as the open
ball with radius $\dist(z,w)/2$ and center the midpoint of the hyperbolic line segment 
between $z$ and $w$. Put differently, it is a ball $B$ with $z,w \in \partial B$ that has the smallest 
possible radius among all such balls.
The {\em Gabriel ball} takes its name from the {\em Gabriel graph}. This is the subgraph of the Delaunay graph 
where we only include those edges with $\Zcal \cap \BGab(z,w) = \emptyset$. 

For $z, w \in \Haa^d$, we define the event 

\begin{equation}\label{eq:Edef} E(z,w) := \left\{\begin{array}{l}\text{there exists an open
ball $B$ such that }\\\text{$z, w \in \partial B$ and $\BGab(z,w) \cap B \cap \Zcal_w = \emptyset$}\end{array}\right\}. 
\end{equation}

\noindent
Comparing to~\eqref{eq:DelEdge}, it is easily seen that if $z,w$ is an edge of the Poisson-Delaunay graph then 
$E(z,w)$ must hold.

\subsection{Upper bounding the probability of an edge.\label{sec:edgeprob}}

Here we establish that the probability of $E(z,w)$ decays doubly exponentially fast in 
$\dist(z,w)$:

\begin{lemma}\label{lem:edgeprob}
There exist $K=K(d), c=c(d)$ such that, 
for all $0<\lambda\leq 1$, all $0 < p \leq 1/2$ and all $z, w \in \Haa^d$:

$$ \Pee_{p,\lambda}( E(z,w) ) \leq 
K \cdot \exp\left[ -c \cdot \lambda \cdot e^{\left(\frac{d-1}{2}\right)\cdot\dist(z,w)}\right]. $$

\end{lemma}

\noindent
In the proof, we'll use the following elementary observation from Euclidean geometry.

\begin{lemma}\label{lem:elemgeoEucl}
If $a \in \eR^d\setminus \{\orig\}$ and an open ball $B$ are such that $a,-a \in \partial B$ then 
there is a linear halfspace $L$ such that

$$ B_{\eR^d}(\orig, \norm{a} ) \cap L \subseteq B_{\eR^d}(\orig, \norm{a} ) \cap B. $$

\end{lemma}

\noindent
For completeness we provide a proof, in Appendix~\ref{sec:elemgeoEucl}.

\begin{proofof}{Lemma~\ref{lem:edgeprob}}
Let $s_0$ be a large constant, to be specified more precisely as we proceed.
We first remark it suffices to prove the statement assuming 
$s:=\lambda \cdot e^{\left(\frac{d-1}{2}\right)\cdot\dist(z,w)} \geq s_0$.
(If there are $K,c$ that work whenever $s \geq s_0$ then the choice  
$K' := \max( K, \exp[ +c\cdot s_0 ]), c':=c$ works 
for all $z,w$.)
We find it convenient to work in the Poincar\'e ball model in the current proof.
Applying a suitable isometry if needed, we can assume without loss of generality 
that the center of $\BGab(z,w)$ coincides with the center $\orig$ of the Poincar\'e ball, and 
hence $z, w$ will be antipodes. Recalling that in the Poincar\'e ball model, 
hyperbolic balls are also Euclidean balls, the previous lemma gives that

$$ E(z,w) \subseteq \left\{\begin{array}{l}\text{there is a linear halfspace $L$ such}\\
\text{that $\Zcal_w \cap \BGab(z,w) \cap L = \emptyset$}\end{array}\right\} =: F(z,w). $$

\noindent 
We note that $N := \left|\Zcal_w \cap \BGab(z,w) \right| \isd \Po(\nu)$ follows a
Poisson distribution with mean 

$$ \nu := \lambda \cdot (1-p) \cdot \mu( \BGab(z,w) ) \geq 
(\lambda/2) \cdot \mu( \BGab(z,w) ) \geq c_1 \cdot s, $$

\noindent 
where $c_1 := \omega_{d-1} /(2^{d+1}(d-1))$. (We use $p\leq 1/2$ for the first inequality. For the 
second inequality we use~\eqref{eq:asympvolball} and that $s \geq s_0$ and we've chosen $s_0$
sufficiently large -- since $\lambda \leq 1$ we have that $\dist(z,w) \geq  (2/(d-1)) \cdot \ln s_0$ can be 
made arbitrarily large.)

Now note that 

\begin{equation}\label{eq:engel}
\Pee( E(z,w) ) \leq \Pee( N < \nu/2 ) + \sum_{n\geq \nu/2} \Pee( F(z,w) | N = n )\Pee( N=n ).
\end{equation}

By the Chernoff bound (Lemma~\ref{cher} above), we have 

\begin{equation}\label{eq:duivel} 
\Pee( N < \nu/2 ) \leq \exp[ - \nu \cdot H(1/2) ] \leq \exp[ - c_1\cdot H(1/2) \cdot s ], 
\end{equation}

\noindent
with $H(x) := x\ln x - x  + 1$ as in Lemma~\ref{cher}.

Conditional on $N=n$, the random set of points $\mathcal{Z}_w\cap \BGab(z,w)$ is 
distributed like an i.i.d.~sample of size $n$ from
an atom-free rotationally invariant distribution. (Namely, the volume measure of the Poincar\'e ball model,
restricted to $\BGab(z,w)$ and renormalized. See~\eqref{eq:Poincareballvolmeas}.)
This gives, for $n\ge \nu/2$:

$$ \begin{array}{rcl} \Pee(F(z,w)\mid N=n) 
& = & \Pee( \Bi(n-1,1/2) \leq d-1 ) \\
& \leq & \Pee( \Bi(n,1/2) \leq d ) \\
& \leq & \exp[ - (n/2) \cdot H(2d/n) ] \\
& \leq & \exp[ - (1/4) \cdot H(1/2) \cdot c_1 \cdot s ], \end{array} $$

\noindent
where the first line is Wendel's formula (stated as Lemma~\ref{lem:wendel}
above), the second line follows by obvious monotonicity (flipping an extra coin cannot increase the 
number of heads by more than one), the third line is the Chernoff bound, and in the last line we 
use that $s\geq s_0$ and $s_0$ is sufficiently large (note $H(x)$ is decreasing on $[0,1]$.)
It now also follows that

\begin{equation}\label{eq:escher} 
\sum_{n\geq \nu/2} \Pee( F(z,w) | N = n )\Pee( N=n ) \leq 
\exp[ - (1/4) \cdot H(1/2) \cdot c_1 \cdot s ] 
\end{equation}

\noindent 
Combining~\eqref{eq:engel},~\eqref{eq:duivel} and~\eqref{eq:escher} gives

$$ \Pee( E(z,w) ) \leq 2 \cdot \exp[ - (1/4) \cdot H(1/2) \cdot c_1 \cdot s ], $$

\noindent
a bound of the sought form.
\end{proofof}

We are actually going to be using the following conditional version of Lemma~\ref{lem:edgeprob}.

\begin{corollary}\label{cor:edgeprob}
There exist $K=K(d), c=c(d), \delta=\delta(d) > 0$ such that, 
for all $0<\lambda\leq 1$, all $0 < p \leq 1/2$, all $z, w \in \Haa^d$, all $n$ and all 
$z_1,\dots,z_n, w_1,\dots,w_n \in \Haa^d$ 
satisfying 

$$\mu\left( \BGab(z,w) \cap \left( \bigcup_{i=1}^n \BGab(z_i,w_i) \right) \right) < 
\delta \cdot \mu\left( \BGab(z,w) \right), $$ 

\noindent
we have 

$$ \Pee_{p,\lambda}( E(z,w) | E(z_1,w_1) \cap \dots \cap E(z_n,w_n) ) \leq 
K \cdot \exp\left[ -c \cdot \lambda \cdot e^{\left(\frac{d-1}{2}\right)\cdot\dist(z,w)}\right]. $$

\end{corollary}

\begin{proof}
For convenience, we write $r := \dist(z,w), B = \BGab(z,w), B_i = \BGab(z_i,w_i)$ and $E = E(z,w), E_i = E(z_i,w_i)$.
We consider a two-stage procedure where first we reveal the status of the 
white Poisson process inside $R := B_1 \cup \dots \cup B_n$ to 
decide whether or not $E_1\cap\dots\cap E_n$ holds and then 
we reveal the status of the white Poisson process on the remainder of $\Haa^d$ to check if $E$ holds as well.
This thought experiment leads to:

$$ \begin{array}{rcl} 
\Pee( E | E_1 \cap \dots \cap E_n ) 
& \leq & \Pee( E | \Zcal_w \cap R = \emptyset ) \\
& = & \Pee( E | \Zcal_w \cap R \cap B = \emptyset ) \\
& \leq & \Pee(E) / \Pee( \Zcal_w \cap R \cap B = \emptyset ). \end{array} $$

\noindent
(The first inequality follows by obvious monotonicity, the second step holds because $E$ depends only 
on what happens inside $B$, and the last inequality follows from the definition of conditional probability.)
Lemma~\ref{lem:edgeprob} now gives

$$ \Pee( E | E_1 \cap \dots \cap E_n ) \leq 
K\cdot\exp\left[ -c\cdot\lambda \cdot e^{\left(\frac{d-1}{2}\right)\cdot r} \right]
\cdot \exp\left[ + \lambda \cdot \mu(R\cap B)\right], $$

\noindent 
with $K,c$ as provided by that lemma. 
Assuming $\mu(R\cap B) \leq \delta\cdot\mu(B)$, the upper bound~\eqref{eq:volballUB} gives

$$ \mu(R\cap B) \leq \delta \cdot \frac{\omega_{d-1}}{2^{d-1}(d-1)} \cdot e^{\left(\frac{d-1}{2}\right)\cdot r}. $$

\noindent
So 

$$ \Pee( E | E_1 \cap \dots \cap E_n ) \leq  
K \cdot \exp\left[ - c' \cdot \lambda \cdot e^{\left(\frac{d-1}{2}\right)\cdot r} \right], $$

\noindent 
with $c' := c - \delta \cdot \frac{\omega_{d-1}}{2^{d-1}(d-1)}$ -- which is $>0$ provided $\delta$ is chosen 
sufficiently small. 
\end{proof}

\subsection{Overlapping balls in hyperbolic $d$-space.\label{sec:overlap}}

The following geometric fact will be used in conjunction with Corollary~\ref{cor:edgeprob}, when we 
need to control the overlap between Gabriel balls of edges.
The statement may appear counterintuitive to readers more familiar with Euclidean geometry than hyperbolic 
geometry, but the proof is not overly involved.

\begin{lemma}\label{lem:elemgeo} 
There exist $K = K(d), c=c(d) > 0$ such that for all 
$r > 0$ and all $x,y \in \Haa^d$:

$$ \mu\left( B(x,r) \cap B(y,r) \right) \leq K \cdot e^{- c \cdot \dist(x,y)} \cdot 
\mu\left( B(x,r) \right). $$

\end{lemma}

\begin{proof}
For notational convenience, we write $s := \dist(x,y)$.
At the expense of worsening (increasing) $K$, we can and do assume that $s \geq s_0$ 
where $s_0$ is a large constant to be determined more precisely as we proceed with the proof.
(If we have a bound of the required form $K \cdot e^{-c\cdot s}$ that works for $s\geq s_0$, then 
replacing $K$ by $K' := \max( K, e^{c\cdot s_0} )$ gives a bound that works for all $s \geq 0$.)

Also note that if $s > 2r$ then $\mu\left( B(x,r) \cap B(y,r) \right) = 0$ and the bound is trivial.
So we can assume $s \leq 2r$. What is more, at the expense of worsening (decreasing) the constant $c$, we can assume that 
$s \leq r/2$.
(In a bit more detail: $\mu\left( B(x,r) \cap B(y,r) \right)$ is non-increasing in $s = \dist(x,y)$. 
So if we have a bound of the form $K \cdot e^{-c s}$ that works for $s \leq r/2$, then 
the bound $K \cdot e^{-(c/4) s}$ works for all $s \leq 2r$. Hence for all $s$.)

Writing %
$$ \begin{array}{l} 
A_1 := \{ z : r \geq \dist(x,z) \geq \dist(y,z) \geq r/2 \}, \\
A_2 := \{ z : r \geq \dist(y,z) \geq \dist(x,z) \geq r/2 \}, 
\end{array} $$

\noindent
we have 

$$ B(x,r) \cap B(y,r) \subseteq B(x,r/2) \cup B(y,r/2) \cup
A_1 \cup A_2. $$

\noindent
The bound~\eqref{eq:volballrminuss} gives

\begin{equation}\label{eq:aap1} 
\mu(B(x,r/2)) = \mu(B(y,r/2))
\leq e^{-((d-1)/2)\cdot r} \cdot \mu(B(x,r)).
\end{equation}

Now let $z \in A_1$ be arbitrary and consider the hyperbolic triangle 
with corners $x,y,z$. 
We let $\alpha := \angle zxy$ be the angle at $x$ and write 
$t := \dist(x,z), u := \dist(z,y)$. 
The hyperbolic cosine rule gives

$$ \begin{array}{rcl} \cosh(u) & = & \cosh(t)\cosh(s) - \cos\alpha\cdot \sinh(t)\sinh(s) \\
    & = & (1-\cos\alpha) \cosh(t)\cosh(s) + \cos\alpha\cdot\left(\cosh(t)\cosh(s) - \sinh(t)\sinh(s)\right) \\
    & = & (1-\cos\alpha) \cosh(t)\cosh(s) + \cos \alpha \cdot \cosh(t-s). 
   \end{array}
$$

\noindent
We have $t \geq u \geq r/2 \geq s$ (as $z \in A_1$), implying

$$ 2 \cosh(t) \geq \cosh(u) - \cos \alpha \cdot \cosh(t-s)
= (1-\cos\alpha) \cosh(t)\cosh(s). $$

\noindent
Dividing both sides by $\cosh(t)\cosh(s)$ gives

$$ 1 - \cos\alpha \leq \frac{2}{\cosh(s)} \leq 4 e^{-s}. $$

\noindent
Recalling our assumption $s\geq s_0$ and having chosen $s_0$ sufficiently large, we must have 
$4 e^{-s} < 1 - 1/\sqrt{2}$. So $0 \leq \alpha < \pi/4$. As $1-\cos\alpha \geq \alpha^2/3$ for $0 \leq \alpha \leq \pi/4$, it follows that

\begin{equation}\label{eq:alphaeend} 
\alpha \leq 3^{3/2}\cdot e^{-s/2}. 
\end{equation}

\noindent
Let us consider the Poincar\'e ball representation of 
$\Haa^d$. Applying a suitable isometry if needed, we can assume without loss of generality that 
$x = \orig$ is the origin. We note that in this situation, we have $z \in B(\orig,r) \cap \kone( y, 3^{3/2} e^{-s/2} )$ 
for all $z \in A_1$, with $\kone(.,.)$ the standard (Euclidean) cone defined above. (See Figure~\ref{fig:helpful}.)

\begin{figure}[!h]
 \begin{center}
  \includegraphics[width=.5\textwidth]{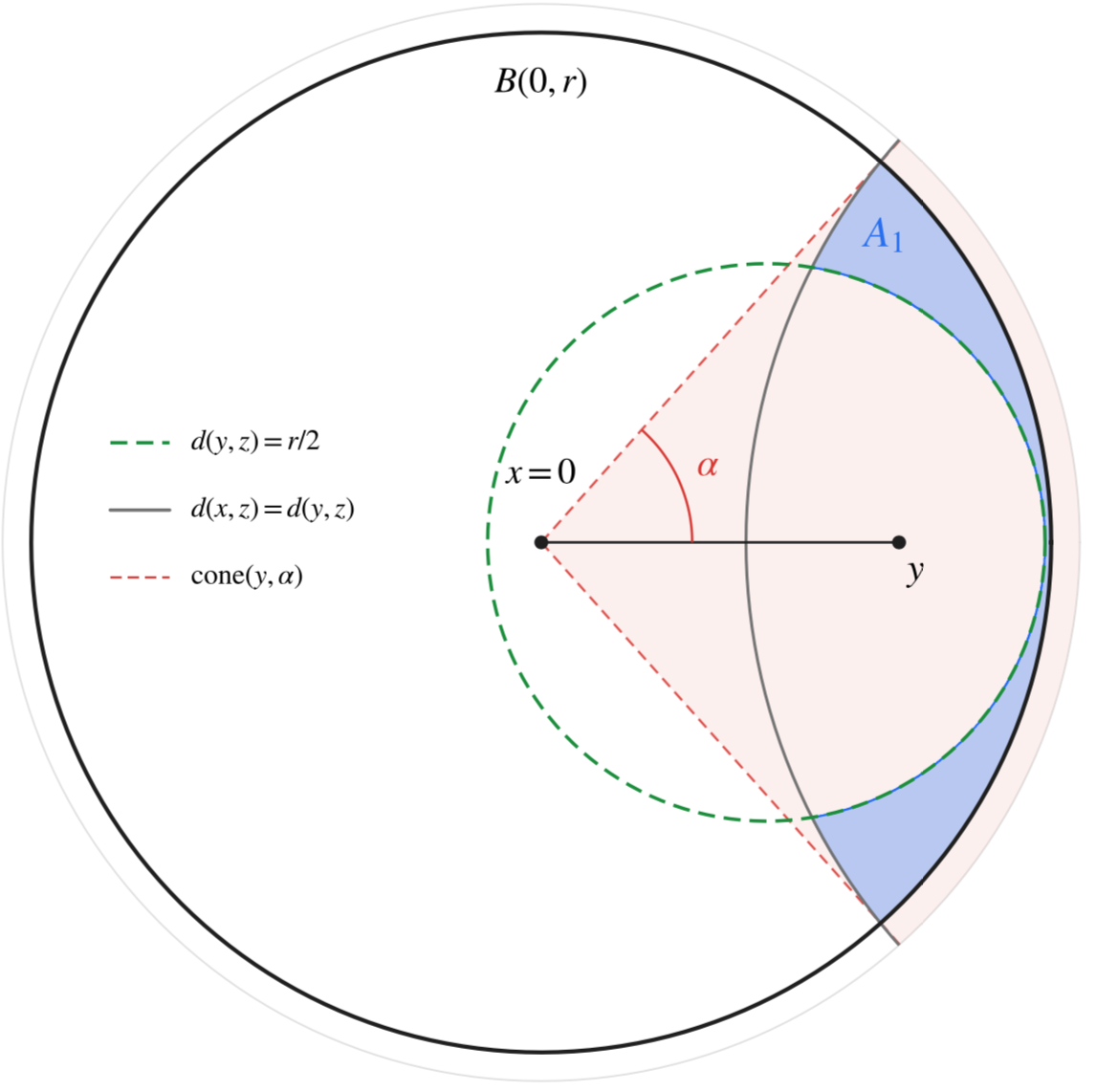}
 \end{center}

 \caption{Illustration of the sets $A_1, \kone(y,\alpha)$.\label{fig:helpful}}
\end{figure}

\noindent
This gives 

$$ \mu(A_1) \leq \mu\left( B(\orig,r) \cap \kone( y, 3^{3/2} e^{-s/2} ) \right) 
= \mu( B(\orig, r ) ) \cdot \Pee( U \in \kone(y, 3^{3/2} e^{-s/2}) ), $$

\noindent 
where $U$ is uniform on $B(\orig,r) $, i.e.~chosen according to the renormalized 
hyperbolic volume measure on $B(\orig,r) $. 
By the (rotational) symmetries of the Poincar\'e ball model, we can invoke Lemma~\ref{lem:kap}, to obtain:

\begin{equation}\label{eq:aap2} 
\mu(A_1) = \mu(A_2) \leq \mu( B(\orig, r ) ) \cdot c \cdot 3^{3(d-1)/2} \cdot e^{-((d-1)/2)\cdot s},  
\end{equation}

\noindent
where $c$ is as provided by Lemma~\ref{lem:kap}.
Combining~\eqref{eq:aap1} and~\eqref{eq:aap2} concludes the proof.
\end{proof}

%
%
%
%
%
%

\subsection{A model with independent edges.\label{sec:gracarmodel}}

In this section, we consider a percolation model whose vertices are the black Poisson points $\Zcal_b$ 
and edges are included independently, with the probability of an edge between $z, w$ 
prescribed by the function $P(z,w)$ which we will now proceed to define.
As mentioned earlier, the model we consider here is {\em almost} a special case of the weight-dependent random connection 
model of Gracar et al.~\cite{GracarEtal}.
Since we need to derive a new result and then transfer it to the 
Poisson-Delaunay graph on $\Haa^d$, we find it convenient to define the model from scratch in a setting 
that stays closer to that of the Poisson-Delaunay graph on $\Haa^d$, as opposed to working with 
the setup of~\cite{GracarEtal} and then having to translate everything back later.

In this subsection, we shall be working exclusively in the Poincar\'e halfspace model of $\Haa^d$.
For $z, w \in \Haa^d$, we set:

$$ \tdist(z,w) := 2 \ln\left(\frac{\norm{(z_1,\dots,z_{d-1})-(w_1,\dots,w_{d-1})}}{2\sqrt{z_dw_d}}\right). $$

\noindent
(With the convention that $\ln(0) = -\infty$.)
We remark that

\begin{equation}\label{eq:tdistdist} 
\dist(z,w) \geq \tdist(z,w),  
\end{equation}

\noindent
for all $z,w$. (This can be seen from~\eqref{eq:distdef} by noting that $\norm{z-w}
\geq \norm{(z_1,\dots,z_{d-1})-(w_1,\dots,w_{d-1})}$ and that $\arcsinh(r) > \ln(r)$ for all $r\geq 0$ -- 
which is perhaps easiest seen by noting that $\sinh(r) = \frac12 (e^r-e^{-r}) < e^r$.)

We fix a constant $s_0>1$ and set, for $s\in [-\infty,\infty)$:

$$ \varphi(s) := 
\begin{cases} 
  1 & \text{ if $s \leq s_0$, and; } \\
  s^{-2} & \text{ otherwise.}
\end{cases}
$$

\noindent
We are now finally ready to define the edge probability $P(z,w)$ by:

\begin{equation}\label{eq:Pdef} P(z,w) := 
\varphi\left( \lambda \cdot e^{\left(\frac{d-1}{2}\right) \cdot \widetilde{\dist}(z,w)} \right). 
\end{equation}

\noindent
As the reader may have noticed $P(z,w)$ is (much) larger than the upper bound on the edge probability in the 
Poisson-Voronoi model we've given in Lemma~\ref{lem:edgeprob} above. 
This will give us some extra wiggle room later on, when we transfer the results from 
this section to the true Poisson-Voronoi model. 

We will call the continuum percolation model we have just described, with vertex set $\Zcal_b$ and 
edges included independently with probabilities according to $P(.,.)$, the {\em independent edge model}.

For convenient future reference, we note that if we write $z = (x,y), w = (u,v) \in \eR^{d-1}\times(0,\infty)$ then 

\begin{equation}\label{eq:tdistnorm}
\exp\left[ \left(\frac{d-1}{2}\right) \cdot \widetilde{\dist}(z,w)\right] = 
\left(\frac{\norm{x-u}}{2\sqrt{y\cdot v}}\right)^{d-1}.
\end{equation}

\noindent
Also for convenient future reference, we point out that 

\begin{equation}\label{eq:cphidef} 
\begin{array}{rcl} 
\int_{\eR^{d-1}} \varphi(\norm{v}^{d-1}) {\dd}v 
& = &  
\omega_{d-2} \cdot \int_0^\infty \varphi( r^{d-1} ) r^{d-2} {\dd}r \\[1.5ex]
& = & 
\omega_{d-2} \cdot \left( \int_0^{s_0^{1/(d-1)}}r^{d-2} {\dd}r 
+ \int_{s_0^{1/(d-1)}}^\infty r^{-d} {\dd r} \right) \\
& = & 
\kappa_{d-1} \cdot \left( s_0 + s_0^{-1} \right) \\
& =: &  
c_\varphi,
\end{array} \end{equation}

\noindent 
is a constant that depends only on the dimension $d$ and the choice of the parameter $s_0$.
(In the first line we've used a switch to multidimensional polar coordinates, see for example~Theorem 15.13 in~\cite{Schilling}, and 
in the penultimate line we use that $\omega_{d-2} = (d-1)\cdot\kappa_{d-1}$.)

A path $z_0, z_1, \dots, z_n$ in the independent edge model is called a 
{\em shortcut-free path} if $P(z_i, z_j ) < 1$ whenever $|i-j|>1$. 
In other words, in a shortcut-free path it is not possible to jump ahead in the path 
using edges that are included in the model automatically because $\tdist$ is too small. Here
``too small'' of course means at most: 

\begin{equation}\label{eq:rho0def} 
\rho_0 := 
(2/(d-1))\cdot(\ln(1/\lambda)+ \ln s_0).
\end{equation}

We now define

$$ 
S_n(h) := 
\Ee_{p,\lambda}^o \left|\left\{\begin{array}{l}
                  \text{shortcut-free paths of length $n$ between $o$ and} \\
                  \text{some $z \in \Zcal_b \cap \left(\eR^{d-1} \times [h,\infty)\right)$ with all internal} \\
                  \text{nodes in $\eR^{d-1} \times (0,h)$}
\end{array} \right\}\right|. $$

The goal of this section is to establish the following result.

\begin{proposition}\label{prop:Snh}
For every $s_0>1$ there exist $c_1 = c_1(s_0), c_2 = c_2(s_0)$ such that

$$ S_n(h) \leq \left(c_1 \cdot p\right)^n \cdot \Pee( \Po( c_2 \cdot \ln h ) < n ), $$

\noindent
for all $0 < p < 1$, all $0<\lambda\leq 1$, all $n \in \eN$ and all $h>1$.
\end{proposition}

We say $z \in \Haa^d$ is {\em underneath} $w \in \Haa^d$ if $z_d < w_d$ and 
$z$ is {\em above} $w$ otherwise. 
For $k \in \eN$ and $h > 1$ we denote by $I_k(h)$ the expected number of (not necessarily shortcut-free) paths of length $k$
between $o$ and some vertex in $\Zcal_b \cap \left(\eR^{d-1} \times [h,\infty)\right)$
such that each vertex is above the previous one and only the last one has $d$-th coordinate $\geq h$, 
in the Palm version of the model wrt.~$o$.

In order to reduce the notational burden in the upcoming arguments, we introduce the notations:

\begin{equation}\label{eq:xigdef}
 \begin{array}{l}
  \xi_n(w_1,\dots,w_n) := 1_{\left\{1 < (w_1)_d < \dots < (w_{n-1})_d < h \leq (w_n)_d\right\}}, \\[2ex]
  f_n(w_1,\dots,w_n) := \xi_n(w_1,\dots,w_n) \cdot P(o,w_1) \cdot P(w_1,w_2) \cdots P(w_{n-1},w_n).
 \end{array}
\end{equation}

\noindent 
We remark that the Mecke formula gives

\begin{equation}\label{eq:IkhMecke}
\begin{array}{rcl} 
I_k(h) 
& := & \displaystyle 
      \Ee_{p,\lambda}^o \left|\left\{\begin{array}{l}
                  \text{paths $o, z_1,\dots,z_k$ with} \\
                  \text{$1 < (z_1)_d < \dots < (z_{k-1})_d < h \leq (z_k)_d$} 
      \end{array} \right\}\right| \\[2ex]
& = &  \displaystyle 
(\lambda p)^k \int_{\Haa^d}\dots\int_{\Haa^d} f_k(z_1,\dots,z_k) \mu({\dd}z_k)\dots\mu({\dd}z_1).
\end{array}
\end{equation}

\noindent
We are now ready for the computations giving the following lemma.

\begin{lemma}\label{lem:Ikh}
For every $s_0>1$ there exist $c_3 = c_3(s_0), c_4 = c_4(s_0)$ such that, 
for all $0 < p < 1$, all $0 < \lambda \leq 1$, all $k \in \eN$ and all $h > 1$ we have 

$$ I_k(h) = \left(c_3 \cdot p\right)^k \cdot 
\Pee\left( \Po\left( c_4 \cdot \ln h \right) = k-1 \right). $$

\end{lemma}

\begin{proof}
For convenience, we introduce some notations. We will write $z_0=o$ and 
$z_i = (x_i,y_i) \in \eR^{d-1}\times(0,\infty)$ for $i=0,\dots,k$.
We also set $V := \{ (y_1,\dots,y_k) \in \eR^k : 1 < y_1 < \dots < y_{k-1} < h \leq y_k \}$.
Combining~\eqref{eq:IkhMecke} and~\eqref{eq:tdistnorm} leads us to:
$$ \begin{array}{rcl} 
I_k(h) 
& = & \displaystyle
(p\lambda)^k \int\dots\int_V (y_1\dots y_k)^{-d} \int_{\eR^{d-1}}\dots\int_{\eR^{d-1}} \\
& & \displaystyle \hspace{8ex}
\prod_{i=1}^k \varphi\left( \lambda \cdot \left(\frac{\norm{x_i-x_{i-1}}}{2\sqrt{y_iy_{i-1}}}\right)^{d-1} \right) 
{\dd}x_k\dots{\dd}x_1 {\dd}y_k\dots{\dd}y_1 \\[2ex]
& = & \displaystyle
(p\lambda)^k \int\dots\int_V (y_1\dots y_k)^{-d} \cdot \prod_{i=1}^k\left(\frac{2^{d-1} (y_iy_{i-1})^{(d-1)/2}}{\lambda}\right) \\
& & \displaystyle \hspace{8ex} \cdot 
\left(\int_{\eR^{d-1}} \varphi(\norm{v}^{d-1}) {\dd}v\right)^k {\dd}y_k\dots{\dd}y_1 \\[2ex]
& = & \displaystyle
\left(p \cdot 2^{d-1} \cdot c_\varphi\right)^k \cdot \int\dots\int_V (y_1\dots y_k)^{-1} \cdot y_k^{-(d-1)/2} 
{\dd}y_k\dots{\dd}y_1,
\end{array} $$

\noindent
where we've used the substitutions 
$v_i := \left(\lambda^{1/(d-1)}/(2\sqrt{y_iy_{i-1}})\right)\cdot (x_i-x_{i-1})$ in the third line.
(If we do these substitutions iteratively starting from $i=k$ down to $i=1$ then we avoid 
having to consider the Jacobian.)

We now apply yet another change of variables. Namely, we set 
$w_1 := \ln y_1$ and $w_i := \ln y_i - \ln y_{i-1}$ for $i=2,\dots,k$.
As $\partial w_i / \partial y_j$ equals $1/y_i$ if $j=i$, and $-1/y_j$ if $j=i-1$, and zero otherwise, 
the corresponding Jacobian is 

$$\left|\det\left(\partial w_i / \partial y_j\right)_{i=1,\dots,k,\atop j=1,\dots,k} \right| = (y_1\dots y_k)^{-1}.$$

\noindent
Hence, writing $W := \{ (w_1,\dots,w_k) \in (0,\infty)^k : w_1+\dots+w_{k-1}<\ln h\leq w_1+\dots+w_k\}$
and noting that $y_k = \exp[w_1+\dots+w_k]$, we obtain

$$ \begin{array}{rcl} 
I_k(h) 
& = & \displaystyle
\left(p \cdot 2^{d-1} \cdot c_\varphi\right)^k \cdot \int\dots\int_W 
e^{ -\left(\frac{d-1}{2}\right) \cdot (w_1+\dots+w_k)}
{\dd}w_k\dots{\dd}w_1 \\[2ex]
& = & \displaystyle
\left(\frac{p \cdot 2^{d} \cdot c_\varphi}{d-1} \right)^k \cdot \int\dots\int_W 
\left(\frac{d-1}{2}\right)^k e^{ -\left(\frac{d-1}{2}\right) \cdot (w_1+\dots+w_k)}
{\dd}w_k\dots{\dd}w_1
\end{array} $$

\noindent
The integrand in the last line we recognise as the joint pdf of $k$ i.i.d.~exponential random variables
$W_1,\dots, W_k$ with common parameter $(d-1)/2$.
Hence, writing $c_3:= (2^{d} \cdot c_\varphi)/(d-1), c_4 := (d-1)/2$ we have 

$$ \begin{array}{rcl} 
I_k(h) 
& = & \displaystyle
(c_3 \cdot p)^k \cdot \Pee( W_1 + \dots + W_{k-1} < \ln h \leq  W_1 + \dots + W_k )
\\
& = & \displaystyle
(c_3 \cdot p)^k \cdot \Pee\left( \Po\left( c_4 \cdot \ln h \right) = k-1 \right), 
\end{array} $$

\noindent
as claimed by the lemma. (For the last line we use the standard elementary equivalence between sums of 
i.i.d.~exponentials and arrivals in one-dimensional homogeneous Poisson processes.)
\end{proof}

For $z,w \in \Haa^d$, we let $U_k(z,w)$ denote the expected number of shortcut-free paths between $z$ and $w$ 
with exactly $k$ internal nodes, all of which are {\em underneath} both $z$ and $w$, in the 
Palm version of the model wrt.~$z,w$. 

We denote

\begin{equation}\label{eq:psidef} 
\begin{array}{l}
\displaystyle 
\psi_n(w_0,\dots,w_{n+1}) := 1_{\left\{\begin{array}{l} 
(w_1)_d,\dots,(w_n)_d < (w_0)_d, (w_{n+1})_d \text{ and}\\
\tdist(w_i,w_j) > \rho_0 \text{ whenever } |i-j|>1\end{array}\right\}}, \\[8ex]
g_n(w_0,\dots,w_{n+1}) := \psi_n(w_0,\dots,w_{n+1}) \cdot \prod_{i=1}^{n+1} P(w_{i-1},w_i).
\end{array}
\end{equation}

\noindent
The Mecke formula gives

\begin{equation}\label{eq:UkMecke} 
\begin{array}{rcl}  
U_k(z,w) 
& := & \displaystyle 
\Ee_{p,\lambda}^{z,w} \left|\left\{
\begin{array}{l}
\text{shortcut-free paths between $z$ and $w$,} \\
\text{with exactly $k$ internal nodes,}\\
\text{all lying underneath $z,w$}
\end{array} \right\}\right| \\[4ex]
& = & \displaystyle 
(\lambda p)^{k} 
\int_{\Haa^d}\dots\int_{\Haa^d} 
g_k(z,z_1,\dots,z_k,w)
\mu({\dd}z_1)\dots\mu({\dd}z_{k}).
\end{array}
\end{equation}

\noindent
We have:

\begin{lemma}[Gracar et al.~\cite{GracarEtal}]
\label{lem:U}
For every $s_0>1$ there exists a $c_5 = c_5(s_0)$ such that, 
for all $0 < p < 1$, all $0<\lambda \leq 1$, all $k \in \eN \cup \{0\}$ and all $z , w \in \Haa^d$ we have 

$$ U_k(z,w) \leq (c_5 \cdot p)^k \cdot P(z,w). $$

\end{lemma}

\noindent
This is essentially Lemmas 2.2 and 2.3 in the work of Gracar et al.~\cite{GracarEtal}.
As mentioned previously the model studied in that paper strictly speaking does not subsume the model considered here, but 
the arguments carry over. Their model also has several parameters that can be tuned allowing for a lot of generality, 
and some amount of translation is needed when checking that their proof works in our setting.
What is more, we've found that the arguments of Gracar et al.~giving Lemma~\ref{lem:U} can be simplified somewhat. 
Gracar et al.~use an elaborate encoding of the paths counted by $U_k$ using 
binary trees, which leads to the Catalan numbers entering into the proof. 
It is possible to show directly  that the standard recursion that defines the Catalan numbers applies to some of the expressions involved, 
which removes the need for the mentioned encoding.
For the convenience of the reader, we include such a simplified proof of Lemma~\ref{lem:U} for the specific 
model we consider here, in Appendix~\ref{sec:U}.

\begin{proofof}{Proposition~\ref{prop:Snh}}
Given a path $z_0, z_1,\dots,z_n$ and $0\leq i\leq n$ we shall say that $z_i$ is a {\em left-to-right maximum} if 
$z_0,\dots,z_{i-1}$ lie underneath $z_i$. 
For $n \in \eN$ and $J \subseteq \{0,1,\dots,n\}$ let 
$S_{n,J}(h)$ denote the expected number of paths $z_0=o,z_1,\dots,z_n$ such that
{\bf a)} $\tdist(z_i,z_j) > \rho_0$ whenever $|i-j|>1$ (the path is shortcut-free), and
{\bf b)} $(z_1)_d, \dots, (z_{n-1})_d < h \leq (z_n)_d$, and {\bf c)} the 
left-to-right maxima are precisely $z_j : j \in J$.
(Note that necessarily $0, n \in J$.)

For the moment, let us fix $n$ and $J$ and let us denote by $\eta(z_0,\dots,z_n)$ the indicator function that 
{\bf a)},{\bf b)},{\bf c)} are satisfied and set 
$h(z_0,\dots,z_n) := \eta(z_0,\dots,z_n) \cdot \prod_{i=1}^n P(z_{i-1},z_i)$.
Let us enumerate $J$ as $j_0=0 < j_1 < \dots < j_k = n$. We remark that

$$ \eta(z_0,\dots,z_n) \leq \xi_{k}(z_{j_1},\dots,z_{j_k}) \cdot \prod_{\ell=1}^k
\psi_{n_\ell}(z_{j_{\ell-1}}, \dots, z_{j_\ell}), $$

\noindent 
where $n_{\ell} := j_{\ell}-j_{\ell-1}-1$ and $\xi, \psi$ are as defined in~\eqref{eq:xigdef},~\eqref{eq:psidef}. Hence also 

\begin{equation}\label{eq:hineq} h(z_0,\dots,z_n) \leq \xi_{k}(z_{j_1},\dots,z_{j_k}) \cdot \prod_{\ell=1}^k
g_{n_\ell}(z_{j_{\ell-1}}, \dots, z_{j_\ell}), 
\end{equation}

\noindent
where $g_{n}$ is as defined in~\eqref{eq:psidef} for $n \geq 1$ and $g_0(z,w) = P(z,w)$.
The Mecke formula gives

$$ \begin{array}{rcl}
   S_{n,J}(h) 
   & = & \displaystyle 
   (\lambda p)^n \int_{\Haa^d}\dots\int_{\Haa^d} h(z_0,\dots,z_n) \mu({\dd}z_n)\dots\mu({\dd}z_1) \\[2ex]
   & \leq & \displaystyle
   (\lambda p)^k \int_{\Haa^d}\dots\int_{\Haa^d} \xi_{k}(z_{j_1},\dots,z_{j_k})  \\
   & & \displaystyle 
   \cdot
   \prod_{\ell=1}^k \left( (\lambda p)^{n_\ell} \int_{\Haa^d}\dots\int_{\Haa^d}
   g_{n_\ell}(z_{j_{\ell-1}}, \dots, z_{j_\ell}) 
   \mu({\dd}z_{j_\ell-1})\dots\mu({\dd}z_{j_{\ell-1}+1})
   \right) \\
   & & \displaystyle 
   \mu({\dd}z_{j_k})\dots\mu({\dd}z_{j_1})  \\
   & = & 
   \displaystyle
   (\lambda p)^k \int_{\Haa^d}\dots\int_{\Haa^d} \xi_{k}(z_{j_1},\dots,z_{j_k})
   \prod_{\ell=1}^k U_{n_\ell}(z_{j_{\ell-1}},z_{j_{\ell}}) \mu({\dd}z_{j_k})\dots\mu({\dd}z_{j_1})  \\
   & \leq & \displaystyle 
   (c_5 \cdot p)^{n_1+\dots+n_k} \cdot (\lambda p)^k \int_{\Haa^d}\dots\int_{\Haa^d} \xi_{k}(z_{j_1},\dots,z_{j_k})
   \prod_{\ell=1}^k P(z_{j_{\ell-1}},z_{j_{\ell}}) \mu({\dd}z_{j_k})\dots\mu({\dd}z_{j_1}) \\
   & = & \displaystyle 
   (c_5 \cdot p)^{n - k} \cdot I_k(h) \\
   & = & \displaystyle 
   (c_5 \cdot p)^{n - k} \cdot (c_3 \cdot p)^k \cdot \Pee\left( \Po\left( c_4 \cdot \ln h \right) = k-1 \right) \\
   & \leq & \displaystyle 
   \left(\max(c_3,c_5) \cdot p\right)^{n} \cdot \Pee\left( \Po\left( c_4 \cdot \ln h \right) < n \right),
   \end{array} $$

\noindent 
where the second line follows from~\eqref{eq:hineq}, the third line uses~\eqref{eq:UkMecke}, the fourth line 
uses Lemma~\ref{lem:U}, the fifth line uses~\eqref{eq:IkhMecke} and the definition of $n_\ell$, and 
the sixth line uses Lemma~\ref{lem:Ikh}.
It follows that 

$$ \begin{array}{rcl} 
S_n(h) 
& = & 
\sum_{J \subseteq \{0,\dots,n\}, \atop J \ni 0,n} S_{n,J}(h) \\[2ex]
& \leq & 
2^n \cdot 
\left(\max(c_3,c_5) \cdot p\right)^{n} \cdot \Pee\left( \Po\left( c_4 \cdot \ln h \right) < n \right) \\[2ex]
& = & 
\left(2 \cdot \max(c_3,c_5) \cdot p\right)^{n} \cdot \Pee\left( \Po\left( c_4 \cdot \ln h \right) < n \right).
\end{array} $$

\noindent 
So the result holds with $c_1 := 2 \max(c_3,c_5), c_2 := c_4$.
\end{proofof}

\subsection{Reduced paths.\label{sec:redpath}}

In this section, we let $s_0 > 1$ and $x_0 \in \eN$ be large constants, to be determined more precisely as we go along, and 
we let $\rho_0$ be as defined in~\eqref{eq:rho0def}.

A sequence of points $z_0,\dots,z_n \in \Zcal_b \cup \{o\}$ is called a {\em reduced path}
if there exists a partition 

$$N \uplus S \uplus J =  \{1,\dots,n\}, $$

\noindent 
such that the conditions {\bf(i)}--{\bf(v)} below are met.
While digesting the upcoming conditions for a reduced path and the lemmas that follow it, it may aid the reader to 
keep in mind that $N,S,J$ stand for ``normal edges'', respectively ``shortcuts'', respectively ``jumps''.
For notational convenience, whenever the sequence $z_0,\dots,z_n$ is clear from the context, we will write 

$$ r_i := \dist(z_{i-1},z_i), \quad B_i := \BGab(z_{i-1},z_i), \quad E_i := E(z_{i-1},z_i)$$ 

\noindent
and we let $c_i$ denote the midpoint of the 
(hyperbolic) line segment between $z_{i-1}$ and $z_i$ (in other words $c_i$ is the center of $B_i = B(c_i,r_i/2)$).
We also define

\begin{equation}\label{eq:Nixdef} 
N_{i,x} := \{ j \in N : j \leq i, \dist(c_j,c_i) < x, |r_j-r_i| < 5x \}. 
\end{equation}

\noindent
We can now state the conditions for a reduced path:

\begin{enumerate}
 \item If $i \in N$ then $E_i$ holds, and;
 \item $i \in S$ if and only if $\tdist(z_{i-1},z_i) \leq \rho_0$, and;
 \item If $i \in J$ then $i-1 \in N$ and there is some $x \in \{x_0,x_0+1,\dots\}$ such that
 $|N_{i-1,x}|\geq x^{4}$ and $r_i \leq r_{i-1} + 10x$.
 \item $\tdist(z_i,z_j) > \rho_0$ whenever $|i-j|>1$, and;
 \item If $i\in N, x \in \{x_0,x_0+1,\dots\}$ are such that $|N_{i,x}|\geq x^{4}$
 then $\dist(z_i,z_j) > r_i + 10x$ for all $j > i+1$.
\end{enumerate}

The name ``reduced path'' refers to the following fact:

\begin{proposition}\label{prop:pathcontainsredpath}
If $z_0,\dots,z_n \in \Zcal_b \cup \{o\}$ is a path in the black subgraph of the Delaunay graph, 
then there is a subsequence $z_{j_0}, z_{j_1}, \dots, z_{j_k}$ with $j_0=0<j_1< \dots<j_k=n$ that is a
reduced path.
\end{proposition}

This can be shown via a straightforward iterative procedure. We spell it out for completeness, 
in Appendix~\ref{sec:pathcontainsredpath}.

We define $R_n(h)$ analogously to $S_n(h)$:

$$
R_n(h) :=  
\Ee_{p,\lambda}^o \left|\left\{\begin{array}{l}
                  \text{reduced paths of length $n$ between $o$ and} \\
                  \text{some $z \in \Zcal_b \cap \left(\eR^{d-1} \times [h,\infty)\right)$ with all internal} \\
                  \text{nodes in $\eR^{d-1} \times (0,h)$}
\end{array} \right\}\right|.
$$

The goal of the remainder of this section is to prove the following result.

\begin{proposition}\label{prop:Rh}
There exist $s_0>1$ and $x_0 \in \eN$ such that, for all $0\leq p \leq 1/2$, all $0<\lambda\leq 1$, all $n \in \eN$ and all $h>1$, 
we have
$R_n(h) \leq S_n(h)$.
\end{proposition}

We will call a sequence of points $z_0,\dots,z_n \in \Haa^d$ for which some 
partition $N \uplus S \uplus J = \{1,\dots,n\}$ exists such that {\bf(ii)},{\bf(iii)},{\bf(iv)},{\bf(v)}
are satisfied a {\em viable sequence}, and  
we will call such a partition a {\em viable partition} (for $z_0,\dots,z_n$). 
Perhaps superfluously, we point out that whether or not $z_0,\dots,z_n$ is viable depends only on 
the geometric constellation of $z_0,\dots,z_n$ and not on the (other points of the) Poisson point process, but 
in order to determine whether {\bf(i)} also holds -- i.e.~the viable sequence is a reduced path -- we'd typically 
have to reveal the white Poisson process $\Zcal_w$ inside some of the Gabriel balls $B_i$.
Also, a given viable sequence might have more than one corresponding viable partition.
We denote by $\Vcal = \Vcal(z_0,\dots,z_n)$ the family of all viable partitions for
$z_0,\dots,z_n$.

For $z_0,\dots,z_n$ a viable sequence, we denote
the probability that there is some $(N,S,J) \in \Vcal$ such that {\bf(i)} is satisfied 
(in other words, that $z_0,\dots,z_n$ is a reduced path) by:

$$ Q(z_0,\dots,z_n) := \Pee_{p,\lambda}\left( \bigcup_{(N,S,J) \in \Vcal}\bigcap_{i \in N} E_i \right). $$

\noindent
The key observation needed for the proof of Proposition~\ref{prop:Rh} is stated in the following lemma.

\begin{lemma}\label{lem:QP}
There exist $s_0 > 1$ and $x_0 \in \eN$ such that, for all $p \leq 1/2$ and $0 < \lambda \leq 1$ and 
every viable sequence $z_0,\dots,z_n$, we have 

$$ Q(z_0,\dots,z_n) \leq \prod_{i=1}^n P(z_{i-1},z_i), $$

\noindent
with $P(.,.)$ as defined by~\eqref{eq:Pdef} above.
\end{lemma}

Before we dive into the proof of Lemma~\ref{lem:QP}, we derive some properties of viable sequences and viable partitions 
that we will rely on.
To ease the burden of notation, for the time being we fix some viable sequence $z_0,\dots,z_n$ and partition 
$(N,S,J) \in \Vcal$.
The following subsets of $N$ will be relevant when we apply Corollary~\ref{cor:edgeprob} later on, and 
we need to control the overlap between the Gabriel balls of the edges.

\begin{equation}\label{eq:Oixdef} 
O_{i,x} := \{ j \in N : \dist(c_j,c_i) < x, r_i - 5x \leq r_j \leq r_i \} 
\end{equation}

\noindent
(Note the difference with the set $N_{i,x}$. We now allow $j$
to be larger than $i$, but we restrict $r_j$ to be less than $r_i$.)

\begin{lemma}\label{lem:Oixcard}
We have
$|O_{i,x}| \leq 16x^{4} + 1$ 
for every $i \in N$ and $x \in \{x_0,x_0+1,\dots\}$. 
\end{lemma}

\begin{proof}
Let us enumerate the elements of $O_{i,x}$ as $j_1 < j_2 < \dots < j_m$, and write $k := 16x^4$ for notational convenience.
We can assume $m \geq k$, as otherwise there is nothing to prove.
We note that

$$ N_{j_k,2x} \supseteq \{j_1,\dots,j_k\}. $$

\noindent
and consequently, by {\bf (v)}, we have $\dist(z_{j_k}, z_\ell) > r_{j_k} + 20x$ for all $\ell > j_{k}+1$.
If $r_\ell \leq r_{i}$ then  

$$ \begin{array}{rcl} 
\dist(c_{\ell},c_i) 
& \geq & 
\dist(z_{\ell},z_{j_k}) - \dist(z_\ell,c_\ell) - \dist(z_{j_k},c_{j_k}) - \dist(c_{j_k},c_i ) \\
& > &    r_{j_k} + 20x - r_\ell/2 - r_{j_k}/2 - x \\ 
& \geq & r_i - 5x + 20x - r_i - x \\
& = & 14x. 
\end{array} $$

We've just established that for every $\ell > j_k+1$ one of $\dist(c_\ell,c_i)>x$ or $r_\ell > r_i$
holds. In particular $\ell \not\in O_{i,x}$.

We conclude that $O_{i,x} \subseteq \{j_1,\dots,j_k\} \cup \{j_k+1\}$ has cardinality at most $k+1$.
\end{proof}

\begin{corollary}\label{cor:overlapdelta}
For every $\delta > 0$ there exists a $x_1=x_1(\delta)$ such that
whenever we choose $x_0 \geq x_1$ then

$$ \mu\left( B_i \cap 
\left( \bigcup_{j \in N, r_j \leq r_i, \atop j \not\in O_{i,x_0}} B_j \right) \right) 
\leq \delta \cdot \mu( B_i ), $$

\noindent 
for all viable sequences $z_0,\dots,z_n$, all viable partitions $(N,S,J) \in \Vcal(z_0,\dots,z_n)$ and all $i \in N$.
\end{corollary}

For clarity, let us emphasize that $x_1$ depends {\em only} on the dimension $d$ and on $\delta$, and in particular {\em not} on 
$\lambda, p, s_0$. 

\begin{proof}
For notational convenience, we write 

$$ A := 
B_i \cap 
\left( \bigcup_{j \in N, r_j \leq r_i, \atop j \not\in O_{i,x_0}} B_j \right), \quad 
A_x := B_i \cap \left( 
\bigcup_{j \in O_{i,x}, \atop \dist(c_j,c_i) \geq x-1} B_j \right). $$

\noindent
If $j \in N \setminus O_{i,x_0}$ and $r_j \leq r_i$ then either {\bf a)} $\dist(c_i,c_j) \geq x_0$ or
{\bf b)} $\dist(c_i,c_j) < x_0$ and $r_j \leq r_i - 5x_0$.
We point out that in case {\bf b)} we have 

\begin{equation}\label{eq:Bjjjj} 
B_j = B(c_j,r_j/2) \subseteq B( c_i, r_i/2 - x_0). 
\end{equation}

\noindent
(Any $w \in B_j$ satisfies 
$\dist(w,c_i) \leq \dist(w,c_j) + \dist(c_i,c_j) < r_j/2 + x_0  
\leq (r_i-5x_0)/2 + x_0 < r_i/2 - x_0$.) 

In case {\bf a)} there is some $x>x_0$ such that $x-1 \leq \dist(c_i,c_j) < x$. 
We distinguish the subcases {\bf a-1)} when $j \in O_{i,x}$ and {\bf a-2)} 
when $r_j < r_i-5x$. Analogously to~\eqref{eq:Bjjjj}, in case {\bf a-2)} we have 

$$ B_j \subseteq B(c_i,r_i/2 - x) \subseteq B(c_i,r_i/2 - x_0). $$

\noindent
We've thus established

$$ A \subseteq B(c_i,r_i/2 - x_0) \cup \bigcup_{x>x_0} A_x. $$

\noindent
By~\eqref{eq:volballrminuss} we have

$$\mu\left( B\left( c_i, r_i/2 - x_0 \right) \right)
\leq e^{-(d-1)x_0} \cdot \mu(B_i)  < (\delta/2) \cdot \mu(B_i), $$

\noindent
provided $x_0 \geq x_1$ and we've chosen $x_1 = x_1(\delta)$ sufficiently large.
By Lemma~\ref{lem:elemgeo}, for each $j \in O_{i,x+1}$ we have 

$$ \mu(B_i \cap B_j) \leq \mu\left( B_i \cap B(c_j,r_i/2) \right) \leq K \cdot e^{- c(x-1)} \cdot \mu(B_i), $$ 

\noindent
where $K$ and $c$ are as provided by Lemma~\ref{lem:elemgeo}.
Invoking Lemma~\ref{lem:Oixcard}, we find 

\begin{equation}\label{eq:Ax} \begin{array}{rcl} 
\displaystyle \mu(A_x)
& \leq & 
|O_{i,x}| \cdot K \cdot e^{-c(x-1)}\cdot \mu(B_i) \\
& \leq & 
 (16x^{4}+1) \cdot K \cdot e^{-c(x-1)} \cdot \mu(B_i) \\
& =: & 
a_x \cdot \mu(B_i), 
\end{array} \end{equation}

\noindent
Noticing that the sum $a_1+a_2+\dots$ is convergent, we see that we can choose $x_1=x_1(\delta)$ such that 
$a_{x_1} + a_{x_1+1} + \dots < \delta/2$.
This gives that, whenever $x_0 \geq x_1$:

$$ \begin{array}{rcl} 
\mu(A) 
& \leq & 
(\delta/2) \cdot \mu(B_i) + \sum_{x\geq x_0} \mu(A_x) \\[1.5ex] 
& \leq &  
(\delta/2) \cdot \mu(B_i) + \left(\sum_{x\geq x_0} a_x \right) \cdot \mu(B_i) \\[1.5ex]
& < & \delta \cdot \mu(B_i). 
\end{array} $$

\end{proof}

\noindent 
For $x \in \{x_0,x_0+1,\dots\}$, we let $J_x$ denote 
the set of ``jumps'' $j \in J$ where $|N_{j-1,x}| \geq x^{4}$ and $r_j \leq r_{j-1}+10x$. In a formula:

$$ J_x := \{ j \in J : |N_{j-1,x}| \geq x^4, r_j \leq r_{j-1}+10x\}. $$ 

\noindent
We also define, for $i \in N$:

\begin{equation}\label{eq:Jixdef} 
J_{i,x} := \left\{ j \in J_x : i \in N_{j-1,x}\right\} 
\end{equation}

\noindent
So in particular $J = \bigcup_{x\geq x_0} J_x$ and $J_x = \bigcup_{i\in N} J_{i,x}$.

\begin{lemma}\label{lem:JixUB}
We have $|J_{i,x}| \leq 1$, 
for all $i \in N, x \in \{x_0,x_0+1,\dots\}$.
\end{lemma}

\begin{proof}
Aiming for a contradiction, suppose there exist $j, k \in J_{i,x}$ with $j < k$.
We have $|r_i - r_{j-1}|, |r_i-r_{k-1}| \leq 5x$ and $\dist(c_i,c_{j-1}), \dist(c_i,c_{k-1}) < x$.
This implies

\begin{equation}\label{eq:paranoot} 
\begin{array}{rcl} 
\dist(z_{j-1},z_{k-1}) 
& \leq & 
\dist(z_{j-1},c_{j-1}) + \dist(z_{k-1},c_{k-1}) + \dist(c_i,c_{j-1}) \\
& &  + \dist(c_i,c_{k-1}) \\
& < & r_{j-1}/2 + r_{k-1}/2 + 2x \\
& \leq & r_{j-1} + 7x.    
\end{array} 
\end{equation}

\noindent
(We use $|r_{j-1}-r_{k-1}| \leq |r_{j-1}-r_i|+|r_{k-1}-r_i| \leq 10x$ in the second inequality.)

Next, we remark that necessarily $k > j+1$, because $k-1 \in N$ by {\bf (iii)}. So {\bf(v)} implies

$$ \dist(z_{j-1},z_{k-1}) > r_{j-1} + 10x, $$

\noindent 
contradicting~\eqref{eq:paranoot}. 
Our initial assumption that $J_{i,x}$ has at least two distinct
elements must be wrong.
\end{proof}

\begin{corollary}\label{cor:JxUB}
We have $|N| \geq x^4 \cdot |J_x|$ for all $x \in \{x_0,x_0+1,\dots\}$.
\end{corollary}

\begin{proof}
We fix an arbitrary $x \in \{x_0,x_0+1,\dots\}$ and define the auxiliary bipartite graph $G$ by:

$$V(G) = N \cup J_x, \quad  E(G) = \{ ij : j \in J_x, i \in N_{j-1,x} \}. $$

\noindent
By Lemma~\ref{lem:JixUB}, we have $\text{deg}(i) \leq 1$ for all $i \in N$, while 
$\text{deg}(j) \geq x^4$ for all $j \in J_x$ by definition of $J_x$.
This gives 

$$ |N| \geq \sum_{i\in N} \text{deg}(i) = |E(G)| = \sum_{j\in J_x} \text{deg}(j) \geq x^4 \cdot |J_x|. $$

\end{proof}

\begin{proofof}{Lemma~\ref{lem:QP}}
We let $K, c,\delta$ be as provided by Corollary~\ref{cor:edgeprob} and $x_1=x_1(\delta)$ as provided 
by Corollary~\ref{cor:overlapdelta}.
We now set $x_0 := \max(x_1,10\cdot (d-1))$. 
Having picked $x_0$, we choose $s_0>1$ such that 

\begin{equation}\label{eq:choiceofs0} 
K \cdot e^{-cs} < \left(8 \cdot s^4\right)^{-(16x_0^4+1)}, 
\end{equation}

\noindent
for all $s\geq s_0$. (This is clearly possible since the LHS decays exponentially, while 
the RHS is of the form $\text{const} \cdot s^{-\text{const}}$.)

Let $z_0,\dots,z_n$ be an arbitrary viable sequence. 
The event whose probability $Q$ we need to bound considers all viable partitions $(N,S,J) \in \Vcal(z_0,\dots,z_n)$ 
simultaneously, but for the moment we just fix some arbitrary partition $(N,S,J)$. 
For notational convenience, we write for $i \in N \cup J$:

\begin{equation}\label{eq:sidef} 
s_i := \lambda e^{\left(\frac{d-1}{2}\right) r_i}, \quad q_i := K e^{-c s_i}, \quad p_i := s_i^{-2}. 
\end{equation}

\noindent 
We remark that for $i \in N \cup J$, since $r_i \geq \tdist(z_{i-1},z_i) > \rho_0$, we have 

$$ s_i > \lambda e^{\left(\frac{d-1}{2}\right) \rho_0} = s_0, \quad \text{ and } \quad P(z_{i-1},z_i) \geq p_i. $$

We greedily select a subset $M = \{i_1,\dots,i_m\} \subseteq N$ as follows.
We take an element $i_1 \in N$ that maximizes $r_i$, remove 
$O_{i_1,x_0}$, take an element $i_2$ of the remainder $N \setminus O_{i_1,x_0}$ that maximizes $r_i$, remove
$O_{i_2,x_0}$, and so on. 
(We repeat until no more elements are left.)
Recalling Lemma~\ref{lem:Oixcard}, we see that for every element  of $N$ we've kept, we've
discarded at most $16x_0^4$ elements, all with $q_i$ at least as large. This gives:

$$  
\prod_{i \in M} q_i \leq 
\left( \prod_{i \in N} q_i \right)^{1/(16x_0^4+1)} 
\leq  8^{-|N|} \cdot \prod_{i \in N} s_i^{-4}
= 
8^{-|N|} \cdot \left(\prod_{i\in N} p_i\right)^2 . $$

\noindent
using~\eqref{eq:choiceofs0} for the second inequality.
Since 

$$ \prod_{x\geq x_0} (1-x^{-2}) 
\geq 1 - \sum_{x\geq x_0} x^{-2}
\geq 1 - \sum_{x\geq x_0-1} \frac{1}{x(x+1)} = \frac{x_0-2}{x_0-1} > \frac12, $$

\noindent 
(using $x_0 \geq 10(d-1) > 3$ in the final inequality) it follows that 

$$  
\prod_{i \in M} q_i \leq 4^{-|N|} \cdot \left( \prod_{x\geq x_0} (1-x^{-2}) \right)^{|N|}
\cdot \left( \prod_{i\in N} p_i \right)^2. $$

\noindent 
Recalling that $j-1 \in N$ for each $j \in J$, we find:

\begin{equation}\label{eq:boneless} 
\begin{array}{rcl} 
\displaystyle
\prod_{i \in M} q_i 
& \leq & \displaystyle
2^{-(|N|+|J|)} \cdot \prod_{x\geq x_0} (1-x^{-2})^{|N|} \cdot 
\prod_{i\in N} p_i \cdot \prod_{j\in J} p_{j-1} \\[3ex]
& = & \displaystyle
2^{-(n-|S|)} \cdot \prod_{x\geq x_0} (1-x^{-2})^{|N|} \cdot 
\prod_{i\in N} p_i \cdot \prod_{j\in J} p_{j-1}. 
\end{array} 
\end{equation}

\noindent
Let us write $I_{x_0} := J_{x_0}$ and $I_x := J_x \setminus (J_{x_0} \cup \dots \cup J_{x-1})$ for $x>x_0$. 
Corollary~\ref{cor:JxUB} gives:

\begin{equation}\label{eq:imbaboon} \begin{array}{rcl} 
\displaystyle
\prod_{x\geq x_0} (1-x^{-2})^{|N|} \cdot \prod_{j\in J} p_{j-1}
& \leq & \displaystyle
\prod_{x\geq x_0} (1-x^{-2})^{x^4|I_x|} \cdot \prod_{j \in J} p_{j-1} \\[2ex]
& = & \displaystyle
\prod_{x\geq x_0}\prod_{j \in I_x} (1-x^{-2})^{x^4} p_{j-1} \\[2ex]
& \leq & \displaystyle
\prod_{x\geq x_0}\prod_{j \in I_x} e^{-x^2} p_{j-1} 
\leq 
\prod_{x\geq x_0}\prod_{j \in I_x} e^{-10(d-1)x} p_{j-1} \\[2ex]
& \leq & \displaystyle
\prod_{x\geq x_0}\prod_{j \in I_x} p_{j} 
= 
\prod_{j \in J} p_j,
\end{array} \end{equation}
 
\noindent
where we use that $x \geq 10(d-1)$ in the fourth step, that 
$r_j \leq r_{j-1} + 10x$ implies $p_j \geq e^{-10(d-1)x} p_{j-1}$ in the penultimate step (see~\eqref{eq:sidef}) 
and that $I_{x_0}, I_{x_0+1}, \dots$ is a partition of $J$ in the last step.
Combining~\eqref{eq:boneless} and~\eqref{eq:imbaboon}, we see that 

\begin{equation}\label{eq:karpouzi} 
\prod_{i \in M} q_i \leq 2^{-(n-|S|)}\cdot\prod_{i \in N \cup J} p_i 
\leq 2^{-(n-|S|)}\cdot\prod_{i=1}^n P(z_{i-1},z_i).
\end{equation}

\noindent 
(Since $P(z_{i-1},z_i)$ equals 1 if $i \in S$ and is $\geq p_i$ otherwise.)

Next we point out that 

\begin{equation}\label{eq:meloen} 
\Pee\left(\bigcap_{i\in N} E_i\right) \leq \Pee\left(\bigcap_{i\in M} E_i\right)
= \Pee( E_{i_m} ) \cdot \Pee( E_{i_{m-1}} | E_{i_m} ) \cdot \dots \cdot 
\Pee( E_{i_1} | E_{i_2} \cap \dots \cap E_{i_m} ). 
\end{equation}

\noindent
By construction of $M$ and Corollary~\ref{cor:overlapdelta} we have 
$\mu\left( B_{i_j} \cap \left( B_{i_{j+1}} \cup \dots \cup B_{i_m}\right) \right) < \delta \mu( B_{i_j} )$ for all $1 \leq j \leq m$.
Corollary~\ref{cor:edgeprob} now gives:

$$ \Pee( E_{i_j} | E_{i_{j+1}}\cap\dots\cap E_{i_m} ) \leq q_{i_j}. $$

\noindent 
Filling this into~\eqref{eq:meloen} and recalling~\eqref{eq:karpouzi} gives

$$ \Pee\left(\bigcap_{i\in N} E_i\right) \leq \prod_{i \in M} q_i \leq 2^{-(n-|S|)}\cdot\prod_{i=1}^n P(z_{i-1},z_i). $$

\noindent 
Since $(N,S,J) \in \Vcal(z_0,\dots,z_n)$ was arbitrary, this bound holds for every viable partition. 

We now remark that the set $S$ is in fact uniquely determined by the viable sequence $z_0,\dots,z_n$ because
of {\bf(ii)}. (That is, any two viable partitions for $z_0,\dots,z_n$ agree on the set $S$.) This in particular implies

$$ |\Vcal(z_0,\dots,z_n)| \leq 2^{(n-|S|)}. $$

\noindent
We are now in a position to finally derive:

$$ \begin{array}{rcl} 
Q(z_0,\dots,z_n) 
& \leq & \displaystyle 
\sum_{(N,S,J) \in \Vcal} \Pee\left(\bigcap_{i\in N} E_i\right) \\[2ex]
& \leq & \displaystyle 
\sum_{(N,S,J) \in \Vcal} 2^{-(n-|S|)}\cdot\prod_{i=1}^n P(z_{i-1},z_i) \\[3ex]
& \leq & 
\prod_{i=1}^n P(z_{i-1},z_i). 
\end{array} $$

\end{proofof}

Proposition~\ref{prop:Rh} is an easy consequence of Lemma~\ref{lem:QP} and the Mecke formula.

\begin{proofof}{Proposition~\ref{prop:Rh}}
Let us write $\xi(z_1,\dots,z_n)$ for the indicator function that $o,z_1,\dots,z_n$ is a
viable sequence, and $\eta(z_1,\dots,z_n)$ for the indicator that demand {\bf(iv)} is satisfied (but not necessarily the
other demands for viable sequences), and write 
$R := \left( \eR^{d-1} \times (0,h) \right)^{n-1} \times \left( \eR^{d-1} \times [h,\infty) \right) $.
By the Mecke formula, writing $z_0 := o$ for convenience:

$$ \begin{array}{rcl} 
R_n(h) 
& = & \displaystyle 
\left(\lambda p\right)^n \cdot 
\int\dots\int_R \xi(z_1,\dots,z_n) \cdot Q(z_0,\dots,z_n) \mu({\dd}z_n)\dots\mu({\dd}z_1) \\[2ex]
& \leq & \displaystyle 
\left(\lambda p\right)^n \cdot 
\int\dots\int_R \eta(z_1,\dots,z_n) \cdot \prod_{i=1}^n P(z_{i-1},z_i) \mu({\dd}z_n)\dots\mu({\dd}z_1) \\[2ex]
& = & S_n(h).
\end{array} $$

\end{proofof}

\subsection{Wrapping up the proof of Theorem~\ref{thm:main}.\label{sec:wrapitup}}

Combining the results from the previous two sections, we can now give a quick derivation of the 
following observation.

\begin{corollary}\label{cor:p0exists}
There exists $p_0 > 0$ such that, for all $p \leq p_0$ and $0 < \lambda \leq 1$:

$$ \Pee_{p,\lambda}^o\left( o \pijl \Zcal_b \cap \left(\eR^{d-1} \times [h,\infty)\right) \right) \xrightarrow[h\to\infty]{} 0. $$

\end{corollary}

\begin{proof}
We let $s_0, x_0$ be as provided by Proposition~\ref{prop:Rh} and we let $c_1 = c_1(s_0),c_2 = c_2(s_0)$ 
be as provided by Proposition~\ref{prop:Snh}.
For all $0 < p \leq p_0 := \min(1/(2\cdot c_1), 1/2)$, all $n \in \eN$ and all $h>1$, we have 

$$ R_n(h) \leq S_n(h) \leq 2^{-n} \cdot \Pee\left( \Po\left( c_2 \cdot \ln h \right) < n \right). $$

\noindent
Appealing to Proposition~\ref{prop:pathcontainsredpath}, it follows that 

$$ \begin{array}{rcl} 
\Pee_{p,\lambda}^o\left( o \pijl \Zcal_b \cap \left(\eR^{d-1} \times [h,\infty)\right) \right)
& \leq &  
\sum_{n=1}^\infty R_n(h) \\[2ex]
& \leq & 
\sum_{n=1}^{n_0} \Pee\left( \Po\left( c_2 \cdot \ln h \right) < n \right) 
+ \sum_{n > n_0} 2^{-n} \\[2ex]
& \leq & 
n_0 \cdot \Pee\left( \Po\left( c_2 \cdot \ln h \right) < n_0 \right) + 2^{-n_0}, 
\end{array} $$

\noindent
for any $n_0 \in \eN$.

To conclude the proof, let $\eps>0$ be arbitrary.
First choosing $n_0$ such that $2^{-n_0} < \eps/2$, we can choose $h_0$ such that 
$n_0 \cdot \Pee\left( \Po\left( c_2 \cdot \ln h \right) < n_0 \right) < \eps/2$ for all $h>h_0$.

This gives that $\Pee_{p,\lambda}^o\left( o \pijl \Zcal_b \cap \left(\eR^{d-1} \times [h,\infty)\right) \right) < \eps$ 
for all 
sufficiently large $h$, which is what we needed to show.
\end{proof}

\begin{corollary}\label{cor:p0existssss}
There exists $p_0 > 0$ such that, for all $p \leq p_0$ and $0 < \lambda \leq 1$:

$$ \Pee_{p,\lambda}\left( o \pijl (0,\dots,0,h) \right) \xrightarrow[h\to\infty]{} 0. $$

\end{corollary}

\begin{proof}
We fix $0<\lambda\leq 1$ and $p \leq p_0$ with $p_0$ as provided by the previous corollary. 
For notational convenience, we write 

$$\begin{array}{c} 
E_h := \{ o \pijl (0,\dots,0,h) \}, \quad  F_h := \left\{ o \pijl \Zcal_b \cap \left(\eR^{d-1}\times[\sqrt{h},\infty)\right) \right\}, \\[2ex]
G_h := \left\{\begin{array}{l}
\text{the nucleus of the cell that contains}\\
\text{$(0,\dots,0,h)$ lies in $\eR^{d-1}\times(0,\sqrt{h})$}
\end{array}\right\}.    
\end{array} $$

\noindent
Clearly 

$$ \Pee_{p,\lambda}(E_h) \leq \Pee_{p,\lambda}(F_h) + \Pee_{p,\lambda}(G_h). $$

\noindent
By the previous corollary 

$$ \Pee_{p,\lambda}(F_h) \leq \Pee_{p,\lambda}^o(F_h) \xrightarrow[h\to\infty]{} 0. $$

\noindent
(We use that adding the origin as a black nucleus can only make the event $F_h$ more likely.)
Next we observe that if the event $G_h$ holds then 
the ball around $(0,\dots,0,h)$ of radius $\int_{\sqrt{h}}^h y^{-1}{\dd}y = \frac12\ln h$ has to be devoid of
Poisson points. 
So

$$ \begin{array}{rcl} 
\Pee_{p,\lambda}(G_h) 
& \leq & \exp\left[ - \lambda\cdot\mu\left( B\left( (0,\dots,0,h),\frac12\ln h \right) \right)\right] \\[2ex]
& = & \exp\left[ - (1+o_h(1)) \cdot \lambda \cdot \frac{\omega_{d-1}}{2^{d-1}(d-1)} \cdot h^{(d-1)/2} \right] \xrightarrow[h\to\infty]{} 0. 
\end{array} $$

\noindent
where we've used~\eqref{eq:asympvolball} in the second line.
The result follows.
\end{proof}

To conclude the proof of Theorem~\ref{thm:main} we rely on the following well-known fact.

\begin{proposition}\label{prop:pu} If $p>p_u(\lambda)$ then 
 $\displaystyle \inf_{x,y \in \Haa^d} \Pee_{p,\lambda}( x \pijl y ) > 0$.
\end{proposition}

\noindent
This observation has a short and relatively elementary proof. 
We include it for the convenience of the reader, in Appendix~\ref{sec:pu}.
We should however also mention that Greb\'{\i}k and Recke (\cite{GrebikRecke}, Theorem 6.1) in fact prove the stronger result that 
$p_u(\lambda) = \inf\{ p : \inf_{x,y} \Pee_{p,\lambda}( x \pijl y ) > 0 \}$ 
for Poisson-Voronoi percolation defined over an ambient space from a certain family of 
geometric spaces which includes hyperbolic $d$-space.
(But the argument of Greb\'{\i}k and Recke is much more involved than the proof of the statement we need.)
Earlier, a result for Cayley graphs analogous to that of Greb\'{\i}k and Recke was proved by Lyons and Schramm~\cite{LyonsSchramm}.

Another fact that we will use is the following.

\begin{proposition}\label{prop:pcpos}
For every $\lambda_0 > 0$, we have
$\displaystyle \inf_{\lambda > \lambda_0} p_c(\lambda) > 0$.
\end{proposition}

This also seems to be known, but as far as 
we know it has not been spelled out explicitly in the literature. 
It can be easily derived by combining and slightly adapting existing results in the literature, as follows.
Benjamini and Schramm~\cite{BS2001} have shown that $\lambda \mapsto p_c(\lambda)$ is continuous and strictly positive. 
Strictly speaking they considered only dimension $d=2$ but their arguments carry over to general dimension with very minor adjustments.
A striking recent result of B\"{u}hler et al.~\cite{BDRS25} gives that in any fixed dimension $d \geq 2$, as $\lambda \to \infty$, the critical probability 
$p_c(\lambda)$ tends to the 
critical probability for Euclidean Poisson-Voronoi percolation on $\eR^d$ -- which is non-zero. (This confirms
a conjecture of Hansen and the second author~\cite{HansenMuller2024}.)
In particular, there exist $\eps>0, \lambda_1 > 0$ such that $p_c(\lambda) > \eps$ for all $\lambda > \lambda_1$.
By compactness (and the mentioned results of Benjamini and Schramm) we must have 
$\inf_{\lambda_0\leq \lambda \leq \lambda_1} p_c(\lambda) > 0$, completing the (sketch) proof of Proposition~\ref{prop:pcpos}.

The proof of B\"{u}hler et al.~\cite{BDRS25} is rather involved and it is possible to give a more direct proof of 
Proposition~\ref{prop:pcpos} that is relatively short and elementary. So the argument we just sketched could be seen as a case of 
``shooting a sparrow with a cannon''.
For the benefit of the reader we provide a direct and elementary proof of Proposition~\ref{prop:pcpos}, in Appendix~\ref{sec:pcpos}.

\begin{proofof}{Theorem~\ref{thm:main}}
For $0 < \lambda \leq 1$, Proposition~\ref{prop:pu} and Corollary~\ref{cor:p0existssss} imply that 

$$ \inf_{0<\lambda\leq 1} p_u(\lambda) \geq p_0 > 0. $$

\noindent
As $p_u(\lambda) \geq p_c(\lambda)$ for all $\lambda$, Proposition~\ref{prop:pcpos} now gives

$$ \inf_{\lambda > 0} p_u(\lambda) \geq \min\left( p_0, \inf_{\lambda>1} p_c(\lambda) \right) > 0. $$

\end{proofof}

\subsection{Derivation of Corollary~\ref{cor:ipvt}.\label{sec:ipvt}}

We will use the following minor variation on Corollary~\ref{cor:p0existssss}.

\begin{corollary}\label{cor:p0ipvt} 
There exists $p_0>0$ such that for every $\eps > 0$ there exists a $c=c(\eps)$ such that, 
for all $0 < \lambda \leq 1$, all $p \leq p_0$ and all $z,w \in \Haa^d$ with $\dist(z,w) > (1/(d-1)) \cdot \ln(1/\lambda) + c$:

$$ \Pee_{p,\lambda}^z( z \pijl w ) < \eps. $$
\end{corollary}

We defer the proof -- which is a straightforward adaptation of the proofs of Corollaries~\ref{cor:p0exists} and~\ref{cor:p0existssss}
-- to Appendix~\ref{sec:p0ipvt}.

We leverage the work of D'Achille et al.~\cite{dAchilleIdeal}.
A key construction in their paper is a natural coupling of the IPVT and 
a family of Poisson point processes $\left(\Zcal^{(\lambda)}\right)_{\lambda>0}$ on $\Haa^d$, where 
$\Zcal^{(\lambda)}$ has intensity $\lambda$.

We briefly describe the coupling and some relevant features. 
For $r,\lambda > 0$ we define $r^{(\lambda)}$ to be the unique solution of

\begin{equation}\label{eq:rlambdadef} 
\lambda \cdot \mu\left( B\left(o, r^{(\lambda)} \right)\right) 
= \mu\left( B\left(o, r\right)\right).
\end{equation}

For $x \in \Haa^d$ and $\lambda > 0$ we let $x^{(\lambda)}$ be the (unique) 
point on the geodesic ray starting from the origin $o$ and passing through $x$ with the property that 

\begin{equation}\label{eq:xlambdadef} 
\dist(o,x^{(\lambda)}) = \dist(o,x)^{(\lambda)}.
\end{equation} 

To construct the coupling, we start with a Poisson process 
$\Zcal$ of intensity $1$ on $\Haa^d$, and set 

$$ \Zcal^{(\lambda)} := \left\{ z^{(\lambda)} : z \in \Zcal \right\}. $$

It is straightforward to verify that $\Zcal^{(\lambda)}$ is 
indeed distributed like a Poisson point process of intensity $\lambda$.
We let $C^{(\lambda)}(z)$ denote the Voronoi cell of $z^{(\lambda)}$ in the Voronoi tessellation generated by~$\Zcal^{(\lambda)}$. 
As shown by D'Achille et al.~\cite{dAchilleIdeal}, almost surely, as $\lambda\searrow 0$, for each $z \in \Zcal$ 
the cell $C^{(\lambda)}(z)$
converges to a limit shape $C^{(0)}(z)$ in the sense of the Fell topology. And, the collection 
$\left( C^{(0)}(z) \right)_{z\in\Zcal}$ of these limit shapes forms a tessellation of $\Haa^d$ that is (distributed like) the IPVT.
See Remark 3.5 in~\cite{dAchilleIdeal}. 
We should perhaps mention that we have chosen to present the coupling using 
a different formulation from the one in~\cite{dAchilleIdeal}. 
In Remark 3.5 of~\cite{dAchilleIdeal}, the Poincar\'e ball representation is 
used for the construction, and their points $\tilde{X}_i^{(\lambda)}$ correspond to a Poisson process of 
intensity $\lambda^{d-1}$ rather than $\lambda$.
That our formulation is equivalent to theirs can for instance be seen via a reparameterization of the intensity parameter and applying a suitable isometry 
to map the PPP from the Poincar\'e ball to the Poincar\'e halfspace. (Here ``suitable'' just means the isometry maps the origin 
$\orig=(0,\dots,0)$ of the former to the origin $o=(0,\dots,0,1)$ of the latter.)

We shall make use of the following result from~\cite{dAchilleIdeal}.

\begin{lemma}[\cite{dAchilleIdeal}]\label{lem:DAadjacent}
Almost surely, for every $z, w\in \Zcal$ we have that $C^{(0)}(z), C^{(0)}(w)$ are adjacent if and only if
there exists a $\lambda_0 > 0$ such that $C^{(\lambda)}(z), C^{(\lambda)}(w)$ are adjacent for all $0<\lambda<\lambda_0$.
\end{lemma}

\noindent
We note that the lemma does not occur as stated here in~\cite{dAchilleIdeal}, but rather it is an immediate consequence of 
Proposition 3.12 together with Remark 3.5 in~\cite{dAchilleIdeal}.

We extend the coupling to a coupling of the corresponding percolation models in the obvious consistent way. Namely, for 
a (fixed) $p \in [0,1]$, we independently pick a colour for each $z \in \Zcal$ (black with probability $p$, white with probability $1-p$). 
In each of the tessellations $(C^{(\lambda)}(z))_{z \in \Zcal}$, we colour the cell $C^{(\lambda)}(z)$ with the colour we assigned to $z$.
Lemma~\ref{lem:DAadjacent} has the following consequence.

\begin{corollary}\label{cor:oneimplicationonly}
For every fixed $p \in [0,1]$, almost surely, for every $z, w\in \Zcal$, {\em if} $C^{(0)}(z), C^{(0)}(w)$
are in the same black cluster
{\em then} there exists a $\lambda_0 > 0$ such that $C^{(\lambda)}(z), C^{(\lambda)}(w)$
are in the same black cluster for all $0<\lambda<\lambda_0$.
\end{corollary}

\begin{proof}
If $C^{(0)}(z), C^{(0)}(w)$ are in the same black cluster of the IPVT, then there exists a finite sequence 
$z_0=z, z_1, \dots, z_k = w \in \Zcal$ such that $C^{(0)}(z_{i-1}), C^{(0)}(z_i)$ are adjacent for $i=1,\dots,k$.
Lemma~\ref{lem:DAadjacent} provides $\lambda_{0,1},\dots,\lambda_{0,k}$ such that the cells
$C^{(\lambda)}(z_{i-1}), C^{(\lambda)}(z_i)$ are adjacent for all $0 < \lambda < \lambda_{0,i}$.
So $\lambda_0 := \min(\lambda_{0,1},\dots,\lambda_{0,k})$ does the trick.
\end{proof}

For $u,v \in \Haa^d$ and $\lambda\geq 0$ we denote by $u \overset{(\lambda)}\pijl v$ that $u \pijl v$ holds with respect to the 
tessellation $(C^{(\lambda)}(z))_{z \in \Zcal}$ and our chosen colouring.
For $s, \lambda > 0$, we define 

$$ V_s^{(\lambda)} := 
\left|\left\{ z \in \Zcal^{(\lambda)} \cap B\left(o,s^{(\lambda)}\right) : o \overset{(\lambda)}\pijl z \right\}\right|. $$

Phrased differently, $V_s^{(\lambda)}$ is the number of cells in the Voronoi tessellation wrt.~the intensity $\lambda$ Poisson process $\Zcal^{(\lambda)}$ that belong to the same black cluster as the zero cell (i.e.~the cell that contains the origin) and whose nucleus lies inside the ball of radius $s^{(\lambda)}$ around $o$.

\begin{lemma}\label{lem:Vslambdaabovepu}
If $p > p_u(\text{IPVT})$ then, for every (fixed) $s>0$:

$$ \lim_{\lambda\searrow 0} \Ee V_s^{(\lambda)} = p^2 \cdot \mu\left( B(o,s) \right) 
+ p(1-p)\cdot\left(1 - e^{-\mu\left( B(o,s) \right)}\right). $$

\end{lemma}

\begin{proof}
We first remark that, by the definition of the function $r \mapsto r^{(\lambda)}$ via~\eqref{eq:rlambdadef} and by construction of the coupling:

$$ V_s^{(\lambda)} = \left|\left\{ w \in \Zcal \cap B(o,s) : o \overset{(\lambda)}\pijl w^{(\lambda)} \right\}\right|. $$

\noindent
Moreover, we have $\dist(o,x) < \dist(o,y)$ if and only if 
$\dist(o,x^{(\lambda)}) < \dist(o,y^{(\lambda)})$. 
This implies that if $z \in \Zcal$ is closest to $o$ among all points of $\Zcal$, then 
$z^{(\lambda)}$ is closest to $o$ among all points of $\Zcal^{(\lambda)}$.
The origin $o$ is always in the cell of this closest point and in particular the colour of $o$ stays fixed as we vary $\lambda$.

Let $E$ denote the event that the origin is coloured black, and define

$$ W := \left| \Zcal_b \cap B(o,s) \right| \cdot 1_E. $$

For notational convenience, let us write $U := \left|\Zcal \cap B(o,s)\right|$.
For each $n \geq 1$ we have

$$ \Ee\left( W | U = n \right) = n p^2 + p(1-p). $$

This can be seen as follows. Imagine we reveal the positions $z_1,\dots,z_n$ of the points of $\Zcal \cap B(o,s)$.
One of these, without loss of generality $z_1$, is the nucleus of the zero cell. 
If $z_1$ is coloured white, then $W=0$. If $z_1$ is coloured black, then each of the $n-1$ remaining points has 
probability $p$ of also being coloured black. So the expectation conditioned on ($\Zcal \cap B(o,s) = \{z_1,\dots, z_n\}$ and) the 
event that $z_1$ is black is $1 + (n-1)p$, and hence the expectation not conditioning 
on the colour of $z_1$ (but conditioned on $\Zcal \cap B(o,s) = \{z_1,\dots, z_n\}$)
must be $p \cdot (1+ (n-1) p) = n p^2 + p(1-p)$.

We arrive at 

$$ \begin{array}{rcl} \Ee W & = & \displaystyle \sum_{n=0}^\infty \Ee\left( W | U= n \right) \cdot 
\Pee( U = n ) 
= p^2 \cdot \Ee U + p(1-p) \cdot \Pee( U > 0 ) \\
& = & \displaystyle p^2 \cdot \mu\left( B(o,s) \right) 
+ p(1-p)\cdot\left(1 - e^{-\mu\left( B(o,s) \right)}\right), 
\end{array} $$

\noindent 
where we use that $U$ follows a Poisson distribution with mean $\mu(B(o,s))$.

For all $\lambda, M > 0$, we must have

\begin{equation}\label{eq:VWineqs} W \geq V_s^{(\lambda)} 
\geq W - M \cdot 1_{\{W \neq V_s^{(\lambda)}\}} - W \cdot 1_{\{W > M\}}.
\end{equation}

\noindent
Taking expectations, we get 

$$ \Ee W \geq \Ee V_s^{(\lambda)} \geq \Ee W - M \cdot \Pee( W \neq V_s^{(\lambda)}) - \sum_{n>M} n \cdot \Pee( W=n ). $$

\noindent
Since $W$ has finite expectation, for every $\eps>0$, we can pick an $M$ such that 
$\sum_{n>M} n \cdot \Pee( W=n ) < \eps/2$.
We note that for $p>p_u(\text{IPVT})$ all black cells of the IPVT belong to a single cluster
almost surely. (A subtle point is that the definition of $p_u(\text{IPVT})$ -- as the infimum of all $p>0$ for which the probability that percolation on the IPVT has exactly one infinite black cluster is positive --
seemingly only guarantees some $q \leq p$ such that, when the percolation parameter is set to $q$, there is a.s.~a unique infinite cluster. 
As each cell is adjacent to infinitely many other cells, this implies that at parameter value $q$, all black cells 
belong to the same cluster a.s. We now consider
``standard increasing coupling'' of percolation at parameter value $q$ and parameter value $p$, where all 
$q$-black cells are also $p$-black. Under this coupling, each $p$-black cell is adjacent to some $q$-black cell, a.s. 
In particular, all $p$-black cells also belong to the same cluster a.s.)

Since $\Zcal \cap B(o,s)$ is almost surely finite, invoking Corollary~\ref{cor:oneimplicationonly}, for $p>p_u(\text{IPVT})$ there exists 
a strictly positive (but random) $\Lambda_0$ such that 
$z^{(\lambda)} \overset{(\lambda)}\langepijl w^{(\lambda)}$ for all $z, w \in \Zcal_b \cap B(o,s)$ 
and all $0 < \lambda \leq \Lambda_0$.
We point out that

$$\Pee( W \neq V_s^{(\lambda)}) \leq \Pee( \Lambda_0<\lambda ).$$

We fix a (non-random) $\lambda_0$ such that $\Pee( \Lambda_0 < \lambda_0 ) < \eps/(2M)$.
With this choice of $M, \lambda_0$ we find that 

$$ \Ee W \geq \Ee V_s^{(\lambda)} \geq \Ee W - M \cdot \Pee( \Lambda_0 < \lambda ) - \eps/2 
> \Ee W - \eps, $$

\noindent 
for all $0 < \lambda < \lambda_0$. In particular

$$ \Ee W \geq \limsup_{\lambda \searrow 0} \Ee V_s^{(\lambda)} \geq \liminf_{\lambda \searrow 0} \Ee V_s^{(\lambda)} 
\geq \Ee W - \eps. $$

\noindent
Sending $\eps \searrow 0$ completes the proof.
\end{proof}

\begin{lemma}\label{lem:Vslambdabelowp0}
If $p_0$ is as in Corollary~\ref{cor:p0ipvt} and $p \leq p_0$ then 

$$ \lim_{s\to\infty} \sup_{0 < \lambda\leq 1} \frac{\Ee V_s^{(\lambda)}}{\mu(B(o,s))} = 0. $$

\end{lemma}

\begin{proof} 
We pick an arbitrary $\eps>0$ and let $c=c(\eps)$ be as provided by Corollary~\ref{cor:p0ipvt}.
Having another look at the defining relation~\eqref{eq:rlambdadef} for $s^{(\lambda)}$ and the asymptotics~\eqref{eq:asympvolball}
of $\mu(B(o,s))$, we see that 

\begin{equation}\label{eq:slambdaasymp} s^{(\lambda)} = (1/(d-1))\cdot\ln(1/\lambda) + s + o_s(1), 
\end{equation}

\noindent 
where the error term $o_s(1)$ can be taken uniform over $0 < \lambda \leq 1$.
For notational convenience, let us write $c_\lambda := (1/(d-1))\cdot\ln(1/\lambda) + c$.
Using the Mecke formula, we can write:

$$ \begin{array}{rcl} 
\Ee V_s^{(\lambda)}
& = & \displaystyle 
p\cdot \lambda \int_{B\left(o,s^{(\lambda)}\right)} \Pee_{p,\lambda}^w( o \pijl w )\mu({\dd}w) \\
& \leq & \displaystyle 
\lambda \cdot \eps \cdot \mu\left( B\left(o,s^{(\lambda)}\right)\right) + 
\lambda \cdot \mu\left( B\left(o, c_\lambda \right) \right) \\[2ex]
& \leq & \displaystyle 
\lambda \cdot \mu( B(o,s^{(\lambda)} )) \cdot \left( \eps + \exp\left[ c_\lambda - s^{(\lambda)}\right] \right) \\[2ex]
& = & \displaystyle 
\lambda \cdot \mu( B(o,s^{(\lambda)} )) \cdot \left( \eps + \exp\left[ c - s + o_s(1) \right] \right) \\[2ex]
& = & \displaystyle 
\mu\left( B(o,s) \right) \cdot \left( \eps + \exp[ c - s + o_s(1) ] \right), 
\end{array} $$

\noindent 
where the second line uses the choice of $c=c(\eps)$ as provided by Corollary~\ref{cor:p0ipvt}, 
the third line uses~\eqref{eq:volballrminuss}, the fourth line uses~\eqref{eq:slambdaasymp} and the 
last line the definition of $s^{(\lambda)}$.
As the error term $o_s(1)$ is uniform over $0<\lambda\leq 1$, there exists 
$s_0$ such that

$$ \sup_{0<\lambda\leq 1} \frac{\Ee V_s^{(\lambda)}}{\mu\left( B(o,s) \right)} \leq 2\eps, $$

\noindent 
for all $s>s_0$. Sending $\eps \searrow 0$ concludes the proof.
\end{proof}

To conclude, we simply note that 
Corollary~\ref{cor:ipvt} immediately follows by comparing Lemmas~\ref{lem:Vslambdaabovepu} and~\ref{lem:Vslambdabelowp0}.

\subsection*{Acknowledgements}

We thank Matteo d'Achille, Gilles Bonnet, Nicolas Curien, J\'ulia Komj\'athy, Konstantin Recke and R\'eka Szabo for helpful discussions.
We thank Matteo d'Achille for introducing us to the problem, we thank J\'ulia Komj\'athy for pointing us to the 
reference~\cite{GracarEtal} and we thank Nicolas Curien for suggesting Corollary~\ref{cor:ipvt}.
We thank the organizers of the inspiring MFO mini-workshop ``hyperbolic meets stochastic geometry'', where we 
learned of several exciting new developments, and the problem tackled in the current paper was brought to 
our attention.

\subsection*{Statement on the use of artificial intelligence.}

All of the text and all the mathematical contributions in this paper were generated by the authors, without any use 
of artificial intelligence. 
Figure~\ref{fig:sim} was made using the 3D graphics software suite Blender in conjunction with Python.
We did use AI assistance for the Python code.

\bibliographystyle{plain}
\bibliography{ReferencesIrlbeckMuller}

@article { Wendel62,
    AUTHOR = {Wendel, J.G.},
     TITLE = {A problem in geometric probability},
   JOURNAL = {Math. Scand.},
  FJOURNAL = {Mathematica Scandinavica},
    VOLUME = {11},
      YEAR = {1962},
     PAGES = {109--111},
      ISSN = {0025-5521},
   MRCLASS = {60.15 (53.90)},
  MRNUMBER = {146858},
MRREVIEWER = {H. S. M. Coxeter},
       DOI = {10.7146/math.scand.a-10655},
       URL = {https://doi.org/10.7146/math.scand.a-10655},
}

@book { Kingman,
    AUTHOR = {Kingman, J. F. C.},
     TITLE = {Poisson processes},
    SERIES = {Oxford Studies in Probability},
    VOLUME = {3},
      NOTE = {Oxford Science Publications},
 PUBLISHER = {The Clarendon Press, Oxford University Press, New York},
      YEAR = {1993},
     PAGES = {viii+104},
      ISBN = {0-19-853693-3},
   MRCLASS = {60G05 (60G55 60K99)},
  MRNUMBER = {1207584},
MRREVIEWER = {Dietrich Stoyan},
}

@book { SchneiderWeil,
    AUTHOR = {Schneider, R. and Weil, W.},
     TITLE = {Stochastic and integral geometry},
    SERIES = {Probability and its Applications (New York)},
 PUBLISHER = {Springer-Verlag, Berlin},
      YEAR = {2008},
     PAGES = {xii+693},
      ISBN = {978-3-540-78858-4},
   MRCLASS = {60-02 (52A22 60D05 60G55 62M30)},
  MRNUMBER = {2455326},
MRREVIEWER = {V. K. Ohanyan},
       DOI = {10.1007/978-3-540-78859-1},
       URL = {https://doi.org/10.1007/978-3-540-78859-1},
}

@article { BriedenEtal,
    AUTHOR = {Brieden, A. and Gritzmann, P. and Kannan, R.
              and Klee, V. and Lov\'{a}sz, L. and Simonovits, M.},
     TITLE = {Deterministic and randomized polynomial-time approximation of
              radii},
   JOURNAL = {Mathematika},
  FJOURNAL = {Mathematika. A Journal of Pure and Applied Mathematics},
    VOLUME = {48},
      YEAR = {2001},
    NUMBER = {1-2},
     PAGES = {63--105},
      ISSN = {0025-5793},
   MRCLASS = {52B55 (52A20 52A38 68U05 68W20 68W25)},
  MRNUMBER = {1996363},
MRREVIEWER = {Therese C. Biedl},
       DOI = {10.1112/S0025579300014364},
       URL = {https://doi.org/10.1112/S0025579300014364},
}

@book { Schilling,
    AUTHOR = {Schilling, R.L.},
     TITLE = {Measures, integrals and martingales},
 PUBLISHER = {Cambridge University Press, New York},
      YEAR = {2005},
     PAGES = {xii+381},
      ISBN = {978-0-521-61525-9; 0-521-61525-9},
   MRCLASS = {28-01 (28A12 28A20 28A25 60G42 60G46)},
  MRNUMBER = {2200059},
MRREVIEWER = {Peter Eichelsbacher},
       DOI = {10.1017/CBO9780511810886},
       URL = {https://doi.org/10.1017/CBO9780511810886},
}

@book {Penroseboek,
    AUTHOR = {Penrose, M.D.},
     TITLE = {Random geometric graphs},
    SERIES = {Oxford Studies in Probability},
    VOLUME = {5},
 PUBLISHER = {Oxford University Press},
   ADDRESS = {Oxford},
      YEAR = {2003},
     PAGES = {xiv+330},
      ISBN = {0-19-850626-0},
   MRCLASS = {60-02 (05C80 60D05)},
  MRNUMBER = {1986198 (2005j:60003)},
MRREVIEWER = {Ilya S. Molchanov},
       DOI = {10.1093/acprof:oso/9780198506263.001.0001},
       URL = {http://dx.doi.org/10.1093/acprof:oso/9780198506263.001.0001},
}

@book { StoyanKendallMecke87,
    AUTHOR = {Stoyan, D. and Kendall, W.S. and Mecke, J.},
     TITLE = {Stochastic geometry and its applications},
    SERIES = {Wiley Series in Probability and Mathematical Statistics:
              Applied Probability and Statistics},
      NOTE = {With a foreword by D. G. Kendall},
 PUBLISHER = {John Wiley \& Sons, Ltd., Chichester},
      YEAR = {1987},
     PAGES = {345},
}

@article{ BS2001,
	title={Percolation in the hyperbolic plane},
	author={Benjamini, I. and Schramm, O.},
	journal={Journal of the American Mathematical Society},
	volume={14},
	number={2},
	pages={487--507},
	year={2001}
}

@book{ Ratcliffe,
    AUTHOR = {Ratcliffe, J.G.},
     TITLE = {Foundations of hyperbolic manifolds},
    SERIES = {Graduate Texts in Mathematics},
    VOLUME = {149},
   EDITION = {Third},
 PUBLISHER = {Springer, Cham},
      YEAR = {2019},
     PAGES = {xii+800},
      ISBN = {978-3-030-31597-9; 978-3-030-31596-2},
   MRCLASS = {57M50 (20H10 30F40 57K32)},
  MRNUMBER = {4221225},
       DOI = {10.1007/978-3-030-31597-9},
       URL = {https://doi.org/10.1007/978-3-030-31597-9},
}

@article { GracarEtal,
    AUTHOR = {Gracar, P. and L\"uchtrath, L. and M\"orters, P.},
     TITLE = {Percolation phase transition in weight-dependent random
              connection models},
   JOURNAL = {Adv. in Appl. Probab.},
  FJOURNAL = {Advances in Applied Probability},
    VOLUME = {53},
      YEAR = {2021},
    NUMBER = {4},
     PAGES = {1090--1114},
      ISSN = {0001-8678,1475-6064},
   MRCLASS = {60K35 (05C80)},
  MRNUMBER = {4342578},
       DOI = {10.1017/apr.2021.13},
       URL = {https://doi.org/10.1017/apr.2021.13},
}

@unpublished{ BDRS25, 
Author = {B\"{u}hler, T. and Dembin, B. and Radhakrishnan, R.R. and Severo, F.},
title = {High-intensity {V}oronoi percolation on manifolds},
note = {Preprint, 2025. Available from arXiv:2503.21737},
}

@book{MeesterRoy1996,
  author    = {Meester, R. and Roy, R.},
  title     = {Continuum Percolation},
  series    = {Cambridge Tracts in Mathematics},
  volume    = {119},
  publisher = {Cambridge University Press},
  address   = {Cambridge},
  year      = {1996}
}

@book{LastPenrose2017,
  author    = {Last, G. and Penrose, M.D.},
  title     = {Lectures on the Poisson Process},
  series    = {Institute of Mathematical Statistics Textbooks},
  volume    = {7},
  publisher = {Cambridge University Press},
  year      = {2017},
  doi       = {10.1017/9781316104477}
}

@unpublished{ GrebikRecke,
author = {Greb\'{\i}k, J. and Recke, K.},
title = {{P}oisson-{V}oronoi percolation in higher rank},
note = {Preprint, 2025. Available from arXiv:2504.02435},
}

@article { HansenMuller2024,
    AUTHOR = {Hansen, B.T. and M\"uller, T.},
     TITLE = {Poisson-{V}oronoi percolation in the hyperbolic plane with
              small intensities},
   JOURNAL = {Ann. Probab.},
  FJOURNAL = {The Annals of Probability},
    VOLUME = {52},
      YEAR = {2024},
    NUMBER = {6},
     PAGES = {2342--2405},
      ISSN = {0091-1798,2168-894X},
   MRCLASS = {60K35},
  MRNUMBER = {4815974},
MRREVIEWER = {Anatoly\ Yambartsev},
       DOI = {10.1214/24-aop1698},
       URL = {https://doi.org/10.1214/24-aop1698},
}

@article { HansenMuller2022,
    AUTHOR = {Hansen, B.T. and M\"uller, T.},
     TITLE = {The critical probability for {V}oronoi percolation in the
              hyperbolic plane tends to {$1/2$}},
   JOURNAL = {Random Structures Algorithms},
  FJOURNAL = {Random Structures \& Algorithms},
    VOLUME = {60},
      YEAR = {2022},
    NUMBER = {1},
     PAGES = {54--67},
      ISSN = {1042-9832,1098-2418},
   MRCLASS = {82B43 (60K35)},
  MRNUMBER = {4340473},
MRREVIEWER = {Adam\ Matthew\ Bowditch},
       DOI = {10.1002/rsa.21018},
       URL = {https://doi.org/10.1002/rsa.21018},
}

@article { LyonsSchramm,
    AUTHOR = {Lyons, R. and Schramm, O.},
     TITLE = {Indistinguishability of percolation clusters},
   JOURNAL = {Ann. Probab.},
  FJOURNAL = {The Annals of Probability},
    VOLUME = {27},
      YEAR = {1999},
    NUMBER = {4},
     PAGES = {1809--1836},
      ISSN = {0091-1798,2168-894X},
   MRCLASS = {60K35 (60B99 82B43)},
  MRNUMBER = {1742889},
MRREVIEWER = {Olle\ H\"aggstr\"om},
       DOI = {10.1214/aop/1022677549},
       URL = {https://doi.org/10.1214/aop/1022677549},
}

@unpublished{ dAchilleVanishing,
author = {D'Achille, M. and Greb\'{\i}k, J. and Khezeli, A. and Recke, K. and Wilkens, A.},
title = {Vanishing uniqueness thresholds in {V}oronoi percolation on products},
note = {Preprint, 2025, available from arXiv:2511.23317},
}

@article { dAchilleIdeal,
    AUTHOR = {D'Achille, M. and Curien, N. and Enriquez,
              N. and Lyons, R. and \"{U}nel, M.},
     TITLE = {Ideal {P}oisson-{V}oronoi tessellations on hyperbolic spaces},
   JOURNAL = {Ann. Probab.},
  FJOURNAL = {The Annals of Probability},
    VOLUME = {54},
      YEAR = {2026},
    NUMBER = {2},
     PAGES = {846--891},
}

@unpublished{ LiYu,
author = { Li, X. and Liu, Y.},
title = {Sharpness of phase transition for {V}oronoi percolation in hyperbolic space},
note = {Preprint, 2021. Available from arXiv:2111.07276},
}

@article { Vanneuville19,
    AUTHOR = {Vanneuville, H.},
     TITLE = {Annealed scaling relations for {V}oronoi percolation},
   JOURNAL = {Electron. J. Probab.},
  FJOURNAL = {Electronic Journal of Probability},
    VOLUME = {24},
      YEAR = {2019},
     PAGES = {Paper No. 39, 71},
}

@unpublished{ dAchilleThaele,
author = {D'Achille, M. and Th{\"{a}}le, C.}, 
title = {Face volume densities of positive-intensity and ideal {P}oisson--{V}oronoi tessellations in hyperbolic spaces},
note = {Preprint, 2026. Available from arXiv:2606.26049},
}

@unpublished{ MellickEtal,
title = {{P}oisson--{V}oronoi tessellations and fixed price in higher rank},
author = {Fr{\fontencoding{T1}\selectfont\k{a}}czyk, M. and Mellick, S. and Wilkens, A.},
note = {Annals of Mathematics, to appear},
}

@article { BudzinskiEtalCheeger,
    AUTHOR = {Budzinski, T. and Curien, N. and Petri, B.},
     TITLE = {On {C}heeger constants of hyperbolic surfaces},
   JOURNAL = {Invent. Math.},
  FJOURNAL = {Inventiones Mathematicae},
    VOLUME = {242},
      YEAR = {2025},
    NUMBER = {2},
     PAGES = {511--530},
      ISSN = {0020-9910,1432-1297},
   MRCLASS = {53A05 (05C50 30F99)},
  MRNUMBER = {4965870},
}

@article { MellickIndistinguish,
    AUTHOR = {Mellick, S.},
     TITLE = {Indistinguishability of cells for the ideal {P}oisson
              {V}oronoi tessellation},
   JOURNAL = {Proc. Amer. Math. Soc.},
  FJOURNAL = {Proceedings of the American Mathematical Society},
    VOLUME = {154},
      YEAR = {2026},
    NUMBER = {4},
     PAGES = {1589--1596},
      ISSN = {0002-9939,1088-6826},
   MRCLASS = {60G55 (37A20)},
  MRNUMBER = {5042206},
       DOI = {10.1090/proc/17493},
       URL = {https://doi.org/10.1090/proc/17493},
}

@article { CalkaChapronEnriquez,
    AUTHOR = {Calka, P. and Chapron, A. and Enriquez, N.},
     TITLE = {Poisson-{V}oronoi tessellation on a {R}iemannian manifold},
   JOURNAL = {Int. Math. Res. Not. IMRN},
  FJOURNAL = {International Mathematics Research Notices. IMRN},
    VOLUME = {2021},
      YEAR = {2021},
    NUMBER = {7},
     PAGES = {5413--5459},
}

@article { Freedman97,
    AUTHOR = {Freedman, M.H.},
     TITLE = {Percolation on the projective plane},
   JOURNAL = {Math. Res. Lett.},
  FJOURNAL = {Mathematical Research Letters},
    VOLUME = {4},
      YEAR = {1997},
    NUMBER = {6},
     PAGES = {889--894},
}

@article { BalisterBollobasQuas05,
    AUTHOR = {Balister, P. and Bollob\'{a}s, B. and Quas, A.},
     TITLE = {Percolation in {V}oronoi tilings},
   JOURNAL = {Random Structures Algorithms},
  FJOURNAL = {Random Structures \& Algorithms},
    VOLUME = {26},
      YEAR = {2005},
    NUMBER = {3},
     PAGES = {310--318},
}

@article { BalisterBollobas10,
    AUTHOR = {Balister, P. and Bollob\'{a}s, B.},
     TITLE = {Bond percolation with attenuation in high dimensional
              {V}orono\u{\i} tilings},
   JOURNAL = {Random Structures Algorithms},
  FJOURNAL = {Random Structures \& Algorithms},
    VOLUME = {36},
      YEAR = {2010},
    NUMBER = {1},
     PAGES = {5--10},
}

@article { Isokawa00,
AUTHOR = {Isokawa, Y.},
TITLE = {Some mean characteristics of {P}oisson-{V}oronoi and {P}oisson-{D}elaunay tessellations in hyperbolic planes},
JOURNAL = {Bulletin of the Faculty of Education, Kagoshima University. Natural science},
VOLUME = {52},
YEAR = {2000},
PAGES = {11-25},
}

@article { Isokawa3d,
    AUTHOR = {Isokawa, Y.},
     TITLE = {Poisson-{V}oronoi tessellations in three-dimensional
              hyperbolic spaces},
   JOURNAL = {Adv. in Appl. Probab.},
  FJOURNAL = {Advances in Applied Probability},
    VOLUME = {32},
      YEAR = {2000},
    NUMBER = {3},
     PAGES = {648--662},
}

@article { BSconformal,
    AUTHOR = {Benjamini, I. and Schramm, O.},
     TITLE = {Conformal invariance of {V}oronoi percolation},
   JOURNAL = {Comm. Math. Phys.},
  FJOURNAL = {Communications in Mathematical Physics},
    VOLUME = {197},
      YEAR = {1998},
    NUMBER = {1},
     PAGES = {75--107},
}

@article { Elliot,
    AUTHOR = {Benjamini, I. and Paquette, E. and Pfeffer, J.},
     TITLE = {Anchored expansion, speed and the {P}oisson-{V}oronoi
              tessellation in symmetric spaces},
   JOURNAL = {Ann. Probab.},
  FJOURNAL = {The Annals of Probability},
    VOLUME = {46},
      YEAR = {2018},
    NUMBER = {4},
     PAGES = {1917--1956},
}

@article{ Elliot2,
    AUTHOR = {Benjamini, I. and Krauz, Y. and Paquette, E.},
     TITLE = {Anchored expansion of {D}elaunay complexes in real hyperbolic
              space and stationary point processes},
   JOURNAL = {Probab. Theory Related Fields},
  FJOURNAL = {Probability Theory and Related Fields},
    VOLUME = {181},
      YEAR = {2021},
    NUMBER = {1-3},
     PAGES = {197--209},
      ISSN = {0178-8051,1432-2064},
   MRCLASS = {60D05 (52C20 60G55)},
  MRNUMBER = {4341072},
MRREVIEWER = {Stanislav\ Volkov},
       DOI = {10.1007/s00440-021-01076-y},
       URL = {https://doi.org/10.1007/s00440-021-01076-y},
}

@book{ BollobasRiordanboek,
	title={Percolation},
	author={Bollob{\'a}s, B. and Riordan, O.},
	year={2006},
	publisher={Cambridge University Press}
}

@book { Grimmettboek,
    AUTHOR = {Grimmett, G.},
     TITLE = {Percolation},
 PUBLISHER = {Springer-Verlag, Berlin},
      YEAR = {1999},
}

@article{bollobas2006critical,
  title={The critical probability for random {V}oronoi percolation in the plane is 1/2},
  author={Bollob{\'a}s, B. and Riordan, O.},
  journal= {Probability Theory and Related Fields},
  volume={136},
  number={3},
  pages={417--468},
  year={2006},
  publisher={Springer}
}

@article {AhlbergEtAl16,
    AUTHOR = {Ahlberg, D. and Griffiths, S. and Morris, R. and
              Tassion, V.},
     TITLE = {Quenched {V}oronoi percolation},
   JOURNAL = {Adv. Math.},
  FJOURNAL = {Advances in Mathematics},
    VOLUME = {286},
      YEAR = {2016},
     PAGES = {889--911},
}

@article {AhlbergBaldasso18,
    AUTHOR = {Ahlberg, D. and Baldasso, R.},
     TITLE = {Noise sensitivity and {V}oronoi percolation},
   JOURNAL = {Electron. J. Probab.},
  FJOURNAL = {Electronic Journal of Probability},
    VOLUME = {23},
      YEAR = {2018},
     PAGES = {Paper No. 108, 21},
   MRCLASS = {60K35 (60G55 82B43)},
  MRNUMBER = {3878133},
       DOI = {10.1214/18-ejp233},
       URL = {https://doi.org/10.1214/18-ejp233},
}

@article {Duminil19,
    AUTHOR = {Duminil-Copin, H. and Raoufi, A. and Tassion, V.},
     TITLE = {Exponential decay of connection probabilities for subcritical
              {V}oronoi percolation in {$\Bbb{R}^d$}},
   JOURNAL = {Probab. Theory Related Fields},
  FJOURNAL = {Probability Theory and Related Fields},
    VOLUME = {173},
      YEAR = {2019},
    NUMBER = {1-2},
     PAGES = {479--490},
}

@article {Tassion16,
    AUTHOR = {Tassion, V.},
     TITLE = {Crossing probabilities for {V}oronoi percolation},
   JOURNAL = {Ann. Probab.},
  FJOURNAL = {The Annals of Probability},
    VOLUME = {44},
      YEAR = {2016},
    NUMBER = {5},
     PAGES = {3385--3398},
}

@unpublished { Zvavitch,
author = { Zvavitch, A.}, 
title = {The critical probability for {V}oronoi percolation}, 
note = {MSc. thesis, Weizmann Institute of Science, 1996, available from 
\url{http://www.math.kent.edu/~zvavitch/master_version_dvi.zip}},
}

@incollection{CannonFloydKenyonParry1997,
  author    = {Cannon, J.W. and Floyd, W.J. and Kenyon, R. and Parry, W.R.},
  title     = {Hyperbolic Geometry},
  booktitle = {Flavors of Geometry},
  editor    = {Levy, Silvio},
  series    = {Mathematical Sciences Research Institute Publications},
  volume    = {31},
  pages     = {59--115},
  publisher = {Cambridge University Press},
  address   = {Cambridge},
  year      = {1997}
}

@unpublished{BhupatirajuThesis,
author = {Bhupatiraju, S.},
title = {The low-intensity limit of {B}ernoulli-{V}oronoi and {P}oisson-{V}oronoi measures},
note = {PhD thesis, Indiana University. 2019}
}

\appendix

\section{The proofs we postponed.}

\subsection{Derivations of~\eqref{eq:volballUB},~\eqref{eq:asympvolball} and~\eqref{eq:volballrminuss}.\label{sec:ballrminuss}}

We start with~\eqref{eq:volballUB}. 
Since $\sinh s = \frac12 (e^s - e^{-s}) \leq \frac12 e^s$ for all $s$, we have 

$$ 
\mu( B(x,r) ) 
= \omega_{d-1}\cdot\int_0^r \left(\sinh s\right)^{d-1}{\dd}s 
\leq \frac{\omega_{d-1}}{2^{d-1}} \cdot \int_{-\infty}^r e^{(d-1)s}{\dd}s 
= \frac{\omega_{d-1}}{2^{d-1}(d-1)} \cdot e^{(d-1)r}.   
$$

\noindent 
To see that~\eqref{eq:asympvolball} holds, we note that 
we can write $\sinh s = \frac12 e^s \cdot (1-e^{-2s})$, so that 

$$ \begin{array}{rcl} 
\mu( B(x,r) ) 
& \geq & \displaystyle 
(1-e^{-r})^{d-1} \cdot \frac{\omega_{d-1}}{2^{d-1}} \cdot \int_{r/2}^r e^{(d-1)s}{\dd}s\\ 
& = & \displaystyle 
(1-e^{-r})^{d-1} \cdot \frac{\omega_{d-1}}{2^{d-1}(d-1)} \cdot (e^{(d-1)r}-e^{(d-1)r/2}) \\ 
& = &  \displaystyle 
(1-e^{-r})^{d-1} \cdot (1-e^{-r(d-1)/2}) \cdot\frac{\omega_{d-1}}{2^{d-1}(d-1)} \cdot  e^{(d-1)r} \\
& = & \displaystyle 
(1-o_r(1)) \cdot \frac{\omega_{d-1}}{2^{d-1}(d-1)} \cdot e^{(d-1)r}. 
\end{array} $$

\noindent 
Together with~\eqref{eq:volballUB} this gives~\eqref{eq:asympvolball}.

Next we consider~\eqref{eq:volballrminuss}. We first note that if $x>y>0$ then 

$$ \sinh(x-y) = \frac12 \left( e^{x-y} - e^{y-x} \right) \leq e^{-y} \cdot \frac12 \left( e^{x} - e^{-x} \right)
= e^{-y} \cdot \sinh x. $$

\noindent 
Applying the substitution $u := t+s$ we find

$$ \begin{array}{rcl} 
\mu( B(x,r-s) ) 
& = & 
\omega_{d-1} \cdot \int_0^{r-s} \sinh^{d-1} t {\dd}t
=  \omega_{d-1} \cdot \int_s^{r} \sinh^{d-1}(u-s){\dd}u \\
& \leq & 
e^{-(d-1)s} \cdot \omega_{d-1} \cdot \int_s^{r} \sinh^{d-1}(u){\dd}u 
\leq e^{-(d-1)s} \cdot \omega_{d-1} \cdot \int_0^{r} \sinh^{d-1}(u){\dd}u \\
& = & e^{-(d-1)s} \cdot \mu( B(x,r ) ). 
\end{array} $$

\subsection{Proof of Lemma~\ref{lem:elemgeoEucl}.\label{sec:elemgeoEucl}}

\begin{proofof}{Lemma~\ref{lem:elemgeoEucl}}
Let us write $B = B_{\eR^d}(c,r)$. If $c=\orig$ then we must have $B = B_{\eR^d}(\orig, \norm{a} )$ and any 
linear halfspace $L$ will do. So we assume $c \neq \orig$ from now.
We will show that the linear halfspace 

$$L = \{ x \in \eR^d : c^t x > 0 \},$$ 

\noindent 
does the trick.
To see this we first note that $a,-a \in \partial B$ implies

$$ \begin{array}{rcl} r^2 & = & \norm{c-a}^2 = \norm{c}^2 + \norm{a}^2 - 2c^ta \\
& = & \norm{c+a}^2 = \norm{c}^2 + \norm{a}^2 + 2c^ta, \end{array} $$

\noindent
so that $c^t a = 0$ and $r^2 = \norm{c}^2 + \norm{a}^2$.
Now let $x \in B_{\eR^d}(\orig,\norm{a}) \cap L$ be arbitrary. 
We can write $x = \lambda c + \mu v$ with $c,v$ orthogonal and $\lambda,\mu \in \eR$.
By Pythagoras

$$ \norm{x}^2 = \lambda^2 \norm{c}^2 + \mu^2 \norm{v}^2. $$

\noindent
Since $x \in L$ we have $c^t x = \lambda c^tc + \mu c^tv  = \lambda c^tc> 0$, so that $\lambda > 0$. 
Now 

$$ \begin{array}{rcl} 
\norm{x-c}^2 
& = & 
\norm{(\lambda-1)c + \mu v}^2 
= (\lambda-1)^2 \norm{c}^2 + \mu^2 \norm{v}^2 \\
& = & 
(1-2\lambda)\norm{c}^2 + \lambda^2 \norm{c}^2 + \mu^2 \norm{v}^2 
=  (1-2\lambda)\norm{c}^2 + \norm{x}^2 \\
& < & 
\norm{c}^2 + \norm{a}^2 = r^2. 
\end{array} $$

\noindent
using that $\lambda > 0$ and $\norm{x} < \norm{a}$ in the last step.
This shows $x \in B$, which is what we needed to show.
\end{proofof}

\subsection{Proof of Lemma~\ref{lem:U}.\label{sec:U}}

\begin{proofof}{Lemma~\ref{lem:U}}
We'll write $a_k := \frac{1}{k+1} {2k \choose k}$ for the $k$-th Catalan number.
Since $a_k \leq 4^k$, it suffices to prove that 

\begin{equation}\label{eq:induct} 
U_k(z,w) \leq a_k \cdot (c \cdot p)^k \cdot P(z,w), 
\end{equation}

\noindent 
for some constant $c=c(s_0)$ to be determined more precisely during the proof.
(After having proven~\eqref{eq:induct}, the statement will follow with $c_5 := 4\cdot c$.)

We will use induction. The base case $k=0$ is trivial as $U_0(z,w) = P(z,w)$.
Let us thus assume~\eqref{eq:induct} holds for $\ell = 0,\dots,k$ and consider $U_{k+1}(z,w)$.
If $\tdist(z,w) \leq \rho_0$ then there are no shortcut-free paths with $k+1$ internal nodes
and in particular $U_{k+1}(z,w) = 0$. So we assume $\tdist(z,w) > \rho_0$ from now on.
As usual, we write $z_0 = z, z_{k+2}=w$ in the remainder of the proof.

Recalling the definitions~\eqref{eq:psidef} of $\psi_n$ and $g_n$ we remark that, since some $z_\ell$ is the highest 
among $z_1,\dots,z_{k+1}$, we have%
$$ \psi_{k+1}(z_0,\dots,z_{k+2}) \leq \sum_{\ell=1}^{k+1}
\psi_{1}(z_0,z_\ell,z_{k+2}) \cdot \psi_{\ell-1}(z_0,\dots,z_\ell) \cdot 
\psi_{k+1-\ell}(z_\ell,\dots,z_{k+2}), $$

\noindent 
and hence also%
$$ g_{k+1}(z_0,\dots,z_{k+2}) \leq \sum_{\ell=1}^{k+1}
\psi_{1}(z_0,z_\ell,z_{k+2}) \cdot g_{\ell-1}(z_0,\dots,z_\ell) \cdot 
g_{k+1-\ell}(z_\ell,\dots,z_{k+2}), $$
\noindent
Combining this observation with~\eqref{eq:UkMecke} gives:

$$ \begin{array}{rcl} U_{k+1}(z,w)
& \leq & \displaystyle \sum_{\ell=1}^{k+1} (\lambda p)^{k+1} 
\int_{\Haa^d}\dots\int_{\Haa^d} \psi_{1}(z_0,z_\ell,z_{k+2}) \cdot g_{\ell-1}(z_0,\dots,z_\ell) \cdot 
g_{k+1-\ell}(z_\ell,\dots,z_{k+2}) \\
& & \displaystyle \hspace{5ex} 
\mu({\dd}z_1)\dots\mu({\dd}z_{k+1}) \\
& = & \displaystyle \sum_{\ell=1}^{k+1} \lambda p
\int_{\Haa^d} \psi_{1}(z_0,z_\ell,z_{k+2}) 
\cdot \left( (\lambda p)^{\ell-1} \int_{\Haa^d}\dots\int_{\Haa^d}
g_{\ell-1}(z_0,\dots,z_\ell) 
\mu({\dd}z_1)\dots\mu({\dd}z_{\ell-1})
\right)  \\
& & \displaystyle 
\cdot \left( (\lambda p)^{k+1-\ell} \int_{\Haa^d}\dots\int_{\Haa^d} 
g_{k+1-\ell}(z_\ell,\dots,z_{k+2}) 
\mu({\dd}z_{\ell+1})\dots\mu({\dd}z_{k+1})\right) \mu({\dd}z_\ell) \\
& = & \displaystyle 
\displaystyle \sum_{\ell=1}^{k+1} \lambda p
\int_{\Haa^d} \psi_{1}(z_0,z_\ell,z_{k+2}) \cdot U_{\ell-1}(z_0,z_\ell) \cdot U_{k+1-\ell}(z_\ell,z_{k+2}) 
\mu({\dd}z_\ell) \\
& \leq & \displaystyle 
\lambda p \cdot \left( c\cdot p \right)^k \cdot \left( \sum_{\ell=1}^{k+1} a_{\ell-1} a_{k+1-\ell} \right)
\cdot \int_{\Haa^d} \psi_{1}(z_0,z_\ell,z_{k+2}) \cdot P(z_0,z_\ell) \cdot P(z_\ell,z_{k+2}) 
\mu({\dd}z_\ell) \\
& = & \displaystyle 
\lambda p \cdot \left( c\cdot p \right)^k \cdot a_{k+1}
\cdot \int_{\Haa^d} g_1(z_0,z_\ell,z_{k+2}) 
\mu({\dd}z_\ell),
\end{array} $$

\noindent 
where the third line follows from~\eqref{eq:UkMecke}, the fourth line follows from the induction hypothesis, and
the last line 
by recalling that the Catalan numbers satisfy the recursion $a_{k+1} = \sum_{i = 0}^k a_i \cdot a_{k-i}$
and again recalling~\eqref{eq:UkMecke}.

Let us write $z_\ell = (x,y)$ with $x \in \eR^{d-1}$ and $y \in (0,\infty)$, and similarly
$z_0 = (x_0,y_0), z_{k+2} = (x_{k+2},y_{k+2})$.
Using~\eqref{eq:tdistnorm}, the integral in the last line of the last display can now be written as:

$$ \begin{array}{rcl} I 
& := & \displaystyle 
\int_{\Haa^d} g_1(z_0,z_\ell,z_{k+2}) 
\mu({\dd}z_\ell) \\
& = & \displaystyle 
\int_0^{\min(y_0,y_{k+2})} y^{-d} \int_{\eR^{d-1}}
\varphi\left(\lambda\left(\frac{\norm{x-x_0}}{2\sqrt{y_0y}}\right)^{d-1}\right) 
\cdot \varphi\left(\lambda\left(\frac{\norm{x-x_{k+2}}}{2\sqrt{y_{k+2}y}}\right)^{d-1}\right) {\dd}x{\dd}y. 
\end{array} $$

\noindent
By the triangle inequality, we have $\norm{x-x_0} \geq \norm{x_0-x_{k+2}}/2$ or 
$\norm{x-x_{k+2}} \geq \norm{x_0-x_{k+2}}/2$ for every $x \in \eR^{d-1}$.
Since $\varphi$ is non-increasing, it follows that

$$ \begin{array}{rcl} 
I 
& \leq & \displaystyle
    \int_0^{\min(y_0,y_{k+2})} y^{-d} \int_{\eR^{d-1}}
\varphi\left(\lambda\left(\frac{\norm{x-x_0}}{2\sqrt{y_0y}}\right)^{d-1}\right) 
\cdot \varphi\left(\lambda\left(\frac{\norm{x_0-x_{k+2}}}{4\sqrt{y_{k+2}y}}\right)^{d-1}\right) {\dd}x{\dd}y \\
& & \displaystyle 
+ \int_0^{\min(y_0,y_{k+2})} y^{-d} \int_{\eR^{d-1}}
\varphi\left(\lambda\left(\frac{\norm{x_0-x_{k+2}}}{4\sqrt{y_0y}}\right)^{d-1}\right) 
\cdot \varphi\left(\lambda\left(\frac{\norm{x-x_{k+2}}}{2\sqrt{y_{k+2}y}}\right)^{d-1}\right) {\dd}x{\dd}y \\
& =: & I_1 + I_2.
   \end{array}
$$ 

\noindent
Focusing attention on $I_1$ we have 

$$ \begin{array}{rcl}
    I_1 
    & = & \displaystyle 
    \int_0^{\min(y_0,y_{k+2})} y^{-d} \cdot \varphi\left(\lambda\left(\frac{\norm{x_0-x_{k+2}}}{4\sqrt{y_{k+2}y}}\right)^{d-1}\right)
    \cdot \left( \int_{\eR^{d-1}}
\varphi\left(\lambda\left(\frac{\norm{x-x_0}}{2\sqrt{y_0y}}\right)^{d-1}\right) 
 {\dd}x \right) {\dd}y \\
    & = & \displaystyle 
    \int_0^{\min(y_0,y_{k+2})} y^{-d} \cdot \varphi\left(\lambda\left(\frac{\norm{x_0-x_{k+2}}}{4\sqrt{y_{k+2}y}}\right)^{d-1}
    \right) \cdot 2^{d-1} \cdot y_0^{(d-1)/2}\cdot y^{(d-1)/2} \cdot (1/\lambda)  \\
    & & \displaystyle 
    \cdot 
    \left( \int_{\eR^{d-1}} \varphi\left(\norm{v}^{d-1}\right){\dd}v \right)
{\dd}y \\
& = & \displaystyle 
(1/\lambda) \cdot 2^{d-1} \cdot c_\varphi \cdot y_0^{(d-1)/2} \cdot 
\left( \int_0^{\min(y_0,y_{k+2})} y^{-(d+1)/2} \cdot \varphi\left(\lambda\left(\frac{\norm{x_0-x_{k+2}}}{4\sqrt{y_{k+2}y}}\right)^{d-1}
    \right) {\dd}y 
\right)
   \end{array} $$

\noindent 
where we've used the substitution $v := (\lambda^{1/(d-1)}/(2\sqrt{y_0 y}))\cdot (x-x_0)$ in the second line.
Now note that, by definition of $\varphi$ and since we chose $s_0 > 1$, we have 
$\varphi(s) \leq s_0^2 \cdot s^{-2}$ for all $s > 0$.
This gives 

$$ \begin{array}{rcl}
\varphi\left(\lambda\left(\frac{\norm{x_0-x_{k+2}}}{4\sqrt{y_{k+2}y}}\right)^{d-1}
    \right) 
    & \leq & 
    s_0^2 \cdot \left(\lambda\left(\frac{\norm{x_0-x_{k+2}}}{4\sqrt{y_{k+2}y}}\right)^{d-1}
    \right)^{-2} \\
    & = & s_0^2 \cdot 4^{d-1} \cdot \left(\lambda\left(\frac{\norm{x_0-x_{k+2}}}{2\sqrt{y_{0}y_{k+2}}}\right)^{d-1}
    \right)^{-2} \cdot y_0^{-(d-1)} \cdot y^{(d-1)} \\
    & = & s_0^2 \cdot 4^{d-1} \cdot \varphi\left(\lambda\left(\frac{\norm{x_0-x_{k+2}}}{2\sqrt{y_{0}y_{k+2}}}\right)^{d-1}\right) 
    \cdot y_0^{-(d-1)} \cdot y^{(d-1)} \\
    & = & s_0^2 \cdot 4^{d-1} \cdot P(z_0, z_{k+2}) \cdot y_0^{-(d-1)} \cdot y^{(d-1)}, 
\end{array} $$

\noindent
where in the third line we use our assumption $\tdist(z_0,z_{k+2}) > \rho_0$.
(Or, equivalently, using~\eqref{eq:tdistnorm} and the definition~\eqref{eq:rho0def} of
$\rho_0$, we have $\lambda \cdot \left( \norm{x_0-x_{k+2}} / (2\sqrt{y_0y_{k+2}}) \right)^{d-1} > s_0$.)

Filling this into the expression for $I_1$ gives

$$ \begin{array}{rcl} 
I_1 
& \leq &  
P(z_0, z_{k+2}) \cdot 
(1/\lambda) \cdot 8^{d-1} \cdot c_\varphi \cdot
s_0^2 \cdot y_0^{-(d-1)/2} 
\cdot \left( \int_0^{\min(y_0,y_{k+2})} y^{(d-3)/2} {\dd}y \right) \\
& = & 
P(z_0, z_{k+2}) \cdot 
(1/\lambda) \cdot 8^{d-1} \cdot c_\varphi \cdot s_0^2 \cdot y_0^{-(d-1)/2} 
\cdot \left(\frac{2}{d-1}\right) \cdot \min(y_0,y_{k+2})^{(d-1)/2} \\
& \leq & 
P(z_0, z_{k+2}) \cdot (1/\lambda) \cdot 8^{d-1} \cdot c_\varphi \cdot
s_0^2.
\end{array} $$ 

\noindent 
Swapping the role of $z_0,z_{k+2}$ and repeating the computations, we see the same bound applies to $I_2$.
We have thus obtained

$$ \begin{array}{rcl} U_{k+1}(z_0,z_{k+2}) & \leq &  
\lambda p \cdot \left( c\cdot  p \right)^k \cdot a_{k+1}
\cdot\left(I_1 + I_2\right) \\
& \leq &  
\lambda p \cdot \left( c\cdot p \right)^k \cdot a_{k+1} \cdot (1/\lambda) \cdot 2^{3d-2} \cdot c_\varphi 
\cdot s_0^2 \cdot P(z_0,z_{k+2}) \\
& = & 
(2^{3d-2} \cdot c_\varphi \cdot s_0^2) \cdot p \cdot \left( c\cdot p \right)^k \cdot a_{k+1} \cdot P(z_0,z_{k+2}) \\
& = & (c\cdot p)^{k+1} \cdot a_{k+1} \cdot P(z_0,z_{k+2}), 
\end{array} $$

\noindent
where the last line follows by finally revealing our choice of $c := 2^{3d-2} \cdot c_\varphi \cdot s_0^2$.
We've now established~\eqref{eq:induct} by induction. As observed earlier, the lemma statement follows.
\end{proofof}

\subsection{Proof of Proposition~\ref{prop:pathcontainsredpath}.\label{sec:pathcontainsredpath}}

\begin{proofof}{Proposition~\ref{prop:pathcontainsredpath}} 
We are given $z_0,\dots,z_n \in \Zcal_b \cup \{o\}$ such that $E(z_0,z_1),\dots,E(z_{n-1},z_n)$ all hold, and
we determine the subsequence $z_{j_0}, z_{j_1}, \dots, z_{j_k}$ with $j_0=0 < j_1 < \dots < j_k = n$
and nested sequences $N_0 \subseteq N_1 \subseteq \dots \subseteq N_k \subseteq \{1,\dots,k \}$,
$S_0 \subseteq S_1 \subseteq \dots \subseteq S_k \subseteq \{1,\dots,k \}$ and 
$J_0 \subseteq J_1 \subseteq \dots \subseteq J_k \subseteq \{1,\dots,k \}$
as follows.

We set $j_0:=0, N_0 = S_0 = J_0 := \emptyset$. Assuming we've already determined $j_0,\dots,j_i$ 
and $N_0,\dots, N_i, S_0,\dots,S_i, J_0,\dots,J_i$ with $j_i < n$ but we've not determined
$j_{i+1}$ yet, we define $j_{i+1}, N_{i+1}, S_{i+1}, J_{i+1}$ as follows.
We first define:

$$ \begin{array}{rcl}
U_i & := & \{ j > j_i + 1 : \tdist(z_{j_i},z_j) \leq \rho_0 \}, \\[2ex]

V_i & := & \begin{cases} 
            \emptyset & \text{ if $i\not\in N_i$, } \\
\left\{ j > j_i+1 : 
  \begin{array}{l} 
	\text{There exists }x \in \{x_0,x_0+1,\dots\} \text{ such that } \\
	 |\{ k \in N_i : \dist(c_{j_k},c_{j_i}) < x, |r_{j_k}-r_{j_i}|<5 x \}| \geq x^{4} \\
	 \text{and } \dist(z_{j_i}, z_j ) \leq r_{j_i} + 10 x
  \end{array}
\right\} & \text{ otherwise. }
\end{cases}.
\end{array} $$ 

\noindent 
We now set 

\begin{equation}\label{eq:jipluseendef} 
j_{i + 1} := \max( \{j_i+1\} \cup U_i \cup V_i ). 
\end{equation}

\noindent
If $\tdist(z_{j_i},z_{j_{i+1}}) \leq \rho_0$ then we 
set 

$$N_{i+1} := N_i, \quad S_{i+1} := S_i \cup \{i+1\}, \quad J_{i+1} := J_i. $$

\noindent
If $j_{i+1} = j_i + 1$ and $\tdist(z_{j_i},z_{j_{i+1}}) > \rho_0$ then we 
set 

$$N_{i+1} := N_i \cup \{i+1\}, \quad S_{i+1} := S_i, \quad J_{i+1} := J_i. $$

\noindent
If $j_{i+1} \in V_i \setminus U_i$ then we set 

$$N_{i+1} := N_i, \quad S_{i+1} := S_i, \quad J_{i+1} := J_i \cup \{i+1\}. $$

\noindent
Having determined $j_{i+1}, N_{i+1}, S_{i+1}, J_{i+1}$, we proceed to the next iteration 
unless $j_{i+1} = n$ in which case we stop.

To conclude the proof, it remains to verify that the sequence $v_0 = z_{j_0}, v_1 = z_{j_1}, \dots, v_k = z_{j_k}$ 
satisfies the definition of a reduced path, with $N_k, S_k, J_k$ taking the role of the partition $N, S, J$ in the definition.
That {\bf(i)} holds follows from the observation that whenever $i \in N_k$ we have 
$\tdist(z_{j_{i-1}},z_{j_i}) > \rho_0$ and $j_{i-1} = j_{i}-1$, 
so that 
$v_{i-1}, v_{i}$ are consecutive on the original path and in particular $E(v_{i-1},v_i) = E(z_{j_{i}-1},z_{j_i})$ holds.
That {\bf (ii)} and {\bf (iii)} hold also follows by construction (because of the choice of the sets $U_i,V_i$).
Similarly, {\bf(iv)} and {\bf(v)} hold by construction (because we take the maximum in~\eqref{eq:jipluseendef}). 
\end{proofof}

\subsection{Proof of Proposition~\ref{prop:pu}.\label{sec:pu}}

For convenience of the reader, and completeness, we provide the standard elementary derivation of
Proposition~\ref{prop:pu}.

\begin{proofof}{Proposition~\ref{prop:pu}}
Let $p > p_u(\lambda)$ be arbitrary, and write 
$U$ for the event that there is exactly one infinite cluster.
It is not immediate from the 
definition of $p_u$ that $\Pee_{p,\lambda}(U) > 0$.
We are however guaranteed the existence of a $q<p$ such that $\Pee_{q,\lambda}(U) > 0$.
Since the event $U$ is invariant under isometries of $\Haa^d$, we in fact have $\Pee_{q,\lambda}(U) = 1$. 
(See for example~Meester--Roy \cite[Definition~2.3, Proposition~2.6, and the paragraph after
Proposition~2.7]{MeesterRoy1996}. The proof is phrased for homogeneous Poisson processes 
on Euclidean space $\eR^d$, but the same proof applies to pairs of homogeneous Poisson point processes
$(\Zcal_b,\Zcal_w)$ on $\Haa^d$.) 
Write 

$$ c := \Pee_{q,\lambda}( o \pijl \infty ). $$

\noindent
By invariance under isometries, of course $\Pee_{q,\lambda}(x \pijl\infty) = c$, for all $x \in \Haa^d$.
We also have $c>0$, because otherwise 
$\Pee_{q,\lambda}(U)  \leq \Pee_{q,\lambda}( \bigcup_{x \in \Haa^d \cap \Qu^d} \{ x\pijl\infty\} ) 
\leq \sum_{x \in \Haa^d \cap \Qu^d} \Pee_{q,\lambda}( x\pijl\infty ) 
= 0$.
(Here we use that the cells of the Poisson-Voronoi tessellation a.s.~have non-empty interior.)
We can now write, for $x,y \in \Haa^d$ arbitrary: 

$$ \begin{array}{rcl} \Pee_{p,\lambda}( x\pijl y)
& \geq & 
\Pee_{q,\lambda}(x\pijl y ) 
\geq \Pee_{q,\lambda}( x\pijl\infty \text{ and } y\pijl\infty \text{ and } U ) \\[2ex]
& \geq & 
\Pee_{q,\lambda}( x\pijl\infty \text{ and } y\pijl\infty ) - \Pee_{q,\lambda}( U^c ) 
= \Pee_{q,\lambda}( x\pijl\infty \text{ and } y\pijl\infty ) \\[2ex]
& \geq & 
\Pee_{q,\lambda}(x\pijl\infty)\cdot \Pee_{q,\lambda}(y\pijl\infty ) = c^2, 
\end{array} $$

\noindent
where we've used the Harris-FKG inequality in the last line.
(To be precise, we use Harris--FKG for the independent pair of Poisson processes
$(\Zcal_b,\Zcal_w)$ in the following form: events that are increasing
in \(\mathcal Z_b\) and decreasing in \(\mathcal Z_w\) are positively correlated.
This follows from
\cite[Theorem~20.4]{LastPenrose2017}, applied first to the black process and then
to the white process.)
\end{proofof}

\subsection{Proof of Proposition~\ref{prop:pcpos}\label{sec:pcpos}}

\begin{proofof}{Proposition~\ref{prop:pcpos}}
	We write
\begin{equation*}
        V(r):=\mu(B(o,r)).
\end{equation*}
The black and white points are independent
Poisson processes $\mathcal{Z}_b,\mathcal{Z}_w$ of intensities $p\lambda$ and
$(1-p)\lambda$, respectively. We fix $A>0$, to be chosen large, and let $R_A$ be defined by
\begin{equation*}
        \lambda_0 V(R_A)=A .
\end{equation*}
Set
\begin{equation*}
        \Delta_A:=\sup_{0<R\le R_A}\frac{V(9R)}{V(R/2)} .
\end{equation*}
Since $V(r)\asymp r^d$ for $0<r\leq 1$ and
$V(r)\asymp e^{(d-1)r}$ for $r\geq 1$, we have
$\log \Delta_A=O(R_A)$, while, as $A\to\infty$,
$A=\lambda_0V(R_A)\asymp e^{(d-1)R_A}$.
Thus we may choose $A$ so large that 
\begin{equation}\label{prop5eq1}
        \Delta_A e^{-A/4}<1.                                      
\end{equation}
We now fix $\lambda\ge\lambda_0$, and choose $R=R(\lambda)$ by
\begin{equation*}
        \lambda V(R)=A .
\end{equation*}
Then $R\le R_A$. Let $\Ncal$ be a countable maximal $R$-separated subset
of $\mathbb H^d$. Thus the balls $B(v,R)$, $v\in\Ncal$, cover
$\mathbb H^d$. Put a graph structure on $\Ncal$ by joining $v,w$ whenever
\begin{equation*}
        \dist(v,w)\le 8R .
\end{equation*}
This graph has degree at most $\Delta_A$: if $w$ ranges over the neighbours
of $v$, then the balls $B(w,R/2)$ are disjoint and are all contained in
$B(v,9R)$.

We will call $v\in\Ncal$ bad if
\begin{equation*}
        \mathcal{Z}_b\cap B(v,4R)\ne\emptyset
        \qquad\text{or}\qquad
        \mathcal{Z}_w\cap B(v,R)=\emptyset .
\end{equation*}
For $p\le 1/2$,
\begin{equation*}
\begin{aligned}
\mathbb P(v\text{ is bad})
\le p\lambda V(4R)+\exp\{-(1-p)\lambda V(R)\}  
\le pA\,\frac{V(4R)}{V(R)}+e^{-A/2}.
\end{aligned}
\end{equation*}
Let
\begin{equation*}
        C_A:=A\sup_{0<R\le R_A}\frac{V(4R)}{V(R)}<\infty .
\end{equation*}
By (\ref{prop5eq1}), we may choose $p_0\le 1/2$ so small that, with
\begin{equation*}
        q:=p_0C_A+e^{-A/2},
\end{equation*}
we have
\begin{equation}\label{prop5eq2}
        \Delta_A\sqrt q<1.                                    
\end{equation}
Then, for every $p\le p_0$, each vertex of $\Ncal$ is bad with probability
at most $q$. Moreover, badness events at pairwise non-adjacent vertices are
independent, since they depend only on the Poisson processes inside the
corresponding balls $B(v,4R)$, and these balls are pairwise disjoint.

We claim that there is no infinite path made out of bad vertices. Fix $v\in\Ncal$. If the bad
cluster of $v$ is infinite, then for every $n$ it contains a bad path
\begin{equation*}
        v=v_0,v_1,\ldots,v_n
\end{equation*}
which is geodesic in the bad subgraph and no two non-consecutive vertices are adjacent. Hence
\begin{equation*}
        v_0,v_2,v_4,\ldots
\end{equation*}
are pairwise non-adjacent, and their badness events are independent. Therefore
any such fixed path of length $n$ is bad with probability at most
$q^{(n+1)/2}$. Since there are at most $\Delta_A^n$ paths of length $n$
starting at $v$,
\begin{equation*}
        \mathbb P(v\text{ lies in an infinite bad cluster})
        \le
        \Delta_A^n q^{(n+1)/2}
        \le
        \sqrt q\,(\Delta_A\sqrt q)^n .
\end{equation*}
Letting $n\to\infty$ and using (\ref{prop5eq2}), this probability is zero. Since
$\Ncal$ is countable, there is almost surely no infinite bad cluster.

It remains to compare the bad vertices with the black Voronoi region. Let
$x\in\mathbb H^d$ lie in a black cell, and choose $v\in\Ncal$ with
$\dist(x,v)<R$. If $v$ were not bad, then there would be a white point
$w\in B(v,R)$ and no black point in $B(v,4R)$. Hence
\begin{equation*}
        \dist(x,w)<2R,
\end{equation*}
whereas every black point $b$ would satisfy
\begin{equation*}
        \dist(x,b)>3R.
\end{equation*}
This contradicts that $x$ lies in a black cell. Thus every point of the black
region lies in $B(v,R)$ for some bad $v\in\Ncal$.

Consequently, if the black region had an unbounded connected component, then
the bad vertices whose $R$-balls meet this component would form an infinite
connected set in the graph on $\Ncal$. Indeed, a connected set covered by open
balls has connected intersection graph, and intersecting balls $B(v,R)$ and
$B(w,R)$ satisfy $\dist(v,w)<2R<8R$, so $v$ and $w$ are adjacent. This
would give an infinite bad cluster, which we have just ruled out.

Therefore, for every $\lambda\ge\lambda_0$ and every $p\le p_0$, the black
Voronoi region does not percolate, which was what we had to show.
\end{proofof}

\subsection{Proof of Corollary~\ref{cor:p0ipvt}.\label{sec:p0ipvt}}

\begin{proofof}{Corollary~\ref{cor:p0ipvt}}
Of course we can take $z=o$ and $w$ of the form $w=(0,\dots,0,h)$ without loss of generality.
Examining the proof of Corollary~\ref{cor:p0exists}, we see that its computations hold uniformly over 
all $p\leq p_0$ and $0 \leq \lambda \leq 1$. 
In particular, writing $F_s$ for the event $o \pijl \Zcal_b \cap \left(\eR^{d-1} \times [s,\infty)\right)$, 
we see that there exists $h_0 = h_0(\eps)$ such that 
$\Pee_{p,\lambda}^o\left( F_{h_0} \right) < \eps/2$, for 
all $p \leq p_0$ and all $0<\lambda\leq 1$.

Let $G_{h_0,h}$ denote the event that the nucleus of the cell that contains
$(0,\dots,0,h)$ lies in $\eR^{d-1}\times(0,h_0)$.
Arguing as in the proof of Corollary~\ref{cor:p0existssss}

\begin{equation}\label{eq:smaug} 
\begin{array}{rcl} 
\Pee_{p,\lambda}^o\left( G_{h_0,h} \right) & \leq &  
\exp\left[ - \lambda \cdot \mu\left( B(w, \ln(h/h_0)) \right) \right] \\
& = & 
\exp\left[ - (1+o_h(1)) \cdot \lambda \cdot \frac{\omega_{d-1}}{2^{d-1}(d-1)}\cdot (h/h_0)^{d-1}\right]. 
\end{array}
\end{equation}

\noindent 
Let us assume that $\ln h = \dist(z,w) > (1/(d-1)) \cdot \ln(1/\lambda) + c$ for some (large) constant $c=c(\eps)$ to be determined shortly.
We have

$$ (h/h_0)^{d-1} > (1/\lambda) \cdot e^{(d-1)\cdot c} \cdot h_0^{-(d-1)}. $$

\noindent
Filling this into~\eqref{eq:smaug} it is clear there is a $c = c(\eps)$ such that 
$\Pee_{p,\lambda}^o\left( G_{h_0,h} \right) < \eps/2$ whenever $h$ satisfies $\ln h = \dist(z,w) > (1/(d-1)) \cdot \ln(1/\lambda) + c$.
Combining the bounds we see 

$$ \Pee_{p,\lambda}^z( z \pijl w ) \leq \Pee^o_{p,\lambda}( F_{h_0} ) + \Pee^o_{p,\lambda}( G_{h_0,h} ) < \eps, $$

\noindent 
as required.
\end{proofof}

\end{document}